\documentclass[11pt]{amsart}
\usepackage[top=1.2in, bottom=1.5in, left=1.5in, right=1.5in]{geometry}

\usepackage{amsfonts, amssymb, amsmath, enumerate}

\numberwithin{equation}{section}

\def \D {\mathbb{D}}
\def \C {\mathbb{C}}
\def \N {\mathbb{N}}

\def \R {\mathbb{R}}

\def \RR {\mathcal{R}}

\def \vol {{\rm vol}}

\newtheorem{theorem}{Theorem}[section]
\newtheorem{proposition}[theorem]{Proposition}
\newtheorem{corollary}[theorem]{Corollary}
\newtheorem{lemma}[theorem]{Lemma}

\theoremstyle{remark}

\begin{document}

\title[Non-central sections of the $l_1$-ball]{Non-central sections of the $l_1$-ball}

\author{Hermann K\"onig}
\address[Hermann K\"onig]{Mathematisches Seminar\\
Universit\"at Kiel\\
24098 Kiel, Germany
}
\email{hkoenig@math.uni-kiel.de}


\begin{abstract}
We determine the maximal non-central hyperplane sections of the $l_1^n$-ball if the fixed distance of the hyperplane to the origin is between $\frac 1 {\sqrt 3}$ and $\frac 1 {\sqrt 2}$. This adds to a result of Liu and Tkocz who considered the distance range between $\frac 1 {\sqrt 2}$ and $1$. For $n \ge 4$, the maximal sections are parallel to the $(n-1)$-dimensional coordinate planes. We also study non-central sections of the complex $l_1^2$- and $l_\infty^2$-balls, where the formulas are more complicated than in the real case. Also, the extrema are partially different than in the real case.
\end{abstract}

\maketitle

\section{Introduction}

The study of sections of convex bodies is a very active area in convex geometry and geometric tomography. In an important paper, K. Ball \cite {B} found the maximal central hyperplane sections of the $n$-cube. The minimal ones had been determined by Hadwiger \cite{Ha} and Hensley \cite {He}. Ball's result was extended recently by Eskenazis, Nayar and Tkocz \cite {ENT} to $l_p^n$-balls for very large $p>2$. Meyer and Pajor \cite {MP} solved the maximal section problem for the $l_p^n$-balls, if $0 < p <2$, and the minimal one for the $l_1^n$-ball, which was extended by Koldobsky \cite {K} to $l_p^n$-balls for $0 < p < 2$. \\

For non-central sections at a fixed distance $t$ to the origin, not too many results are known. Moody, Stone, Zach and Zvavitch \cite {MSZZ} proved that the maximal sections of the $l_\infty^n$-unit ball at very large distances from the origin $\sqrt{n-1} < t < \sqrt n$ are those perpendicular to the main diagonal. Pournin \cite {P} extended their result to slightly smaller distances $\sqrt{n-2} < t \le \sqrt {n-1}$ if $n \ge 5$. Liu and Tkocz \cite {LT} proved a corresponding result for the $l_1^n$-ball, if the distance $t$ to the origin is between $\frac 1 {\sqrt 2}$ and $1$. Our main result extends their theorem to the range $t \in [\frac 1 {\sqrt 3} , \frac 1 {\sqrt 2}]$ for $n \ge 4$. For $n=3$, the same result holds only if $t \ge 0.6384$. The difficulty finding the maximal non-central sections at this distance range results from the fact that the formula for non-central $l_1^n$-sections at distance $t$ involves alternating terms with removable singularities. For $n = 2, 3$, the extremal sections are known for all $0 \le t \le 1$, cf. Liu, Tkocz \cite {LT} and Brzezinski \cite {Br}. Non-central sections of the regular $n$-simplex are studied by K\"onig in \cite{Ko} and \cite{Ko2}. \\

For $n=2$, we also investigate non-central complex hyperplane sections of the balls $B_1^2(\C)$ and $B_\infty^2(\C)$. The volume formulas are more complicated than in the real case. Obviously, $B_1^2(\C)$ and $B_\infty^2(\C)$ are no longer isometric as they are in the real case. \\

A paper by Nayar and Tkocz \cite {NT} gives an excellent survey on extremal sections of classical convex bodies. \\

For $ n \in \N$, let $B_1^n := \{ x \in \R \ | \ ||x||_1 \le 1 \}$ be the $l_1^n$-ball and consider the standard unit vectors $a^{(k)} := \frac 1 {\sqrt k} (\underbrace{1, \cdots , 1}_k, 0, \cdots ,0) \in S^{n-1} \subset \R^n$ for $1 \le k \le n$. For $a \in S^{n-1} \subset \R^n$ and $t \ge 0$, let
$$A_1(a,t) := \vol_{n-1} ( \{ x \in B_1^n \ | \ <x,a> = t \} ) = \vol_{n-1} ( B_1^n \cap \{ t a + a^\perp \} ) $$
be the {\it parallel section function} of the $l_1^n$-ball, with $A_1(a,t) = 0$ for $t > 1$. For $n \ge 3$ and $\frac 1 {\sqrt 2} \le t < 1$, the maximal sections at distance $t$ from the origin are parallel to a coordinate plane, as shown by Liu, Tkocz \cite {LT}, i.e. $A_1(a,t) \le A_1(a^{(1)},t)$ for all $a \in S^{n-1}$. We extend their result to $\frac 1 {\sqrt 3} \le t \le \frac 1 {\sqrt 2}$ for $n \ge 4$. Our main result is

\begin{theorem}\label{th1}
Let $n \ge 4$ and $\frac 1 {\sqrt 3} \le t \le 1$. Then for all $a \in S^{n-1}$
$$A_1(a,t) \le A_1(a^{(1)},t) = \frac {2^{n-1}}{(n-1)!} \ (1-t)^{n-1} \ . $$
The maximal hyperplane sections are parallel to coordinate planes. \\
For $n=3$, this only holds for $t \in [\sqrt 2 -1 - \sqrt{5-\frac 7 {\sqrt 2}},1] \simeq [0.6384,1]$.
\end{theorem}

The basic formula for the parallel section function in Section 2 is determined by the Fourier transform method, as explained in Koldobsky's book \cite {K}. We later use the same technique for hyperplane sections in the two-dimensional complex case. In Section 3, we explain the strategy of the proof and prove some monotonicity results needed later. The proof of Theorem \ref{th1} is given in Section 4. In Sections 5 and 6 we consider non-central hyperplane sections of the complex balls $B_\infty^2(\C)$ and $B_1^2(\C)$. \\

\section{A formula for the parallel section function of the $l_1^n$-ball}

To prove Theorem \ref{th1}, we need an explicit formula for $A_1(a,t)$, $a \in S^{n-1}$, $t \ge 0$. By symmetry, we may assume for the coordinates of $a$ that $1 \ge a_1 \ge a_2 \ge \cdots \ge a_n \ge 0$ holds. Meyer and Pajor \cite {MP} mention a formula for slabs in $B_1^n$ which should be proven by induction. Differentiating their formula with respect to the width, one obtains

\begin{proposition}\label{prop1}
Fix $0 \le t \le 1$. Let $n \in \N$ and $a \in S^{n-1}$ with $1 > a_1 > a_2 > \cdots > a_n >0$. Put $x_+ = max (x, 0)$ for all $x \in \R$. Then
\begin{equation}\label{eq2.1}
A_1(a,t) = \frac {2^{n-1}}{(n-1)!} \sum_{j=1}^n \frac {a_j^{n-2} (a_j-t)_+^{n-1}} {\prod_{k=1, k \ne j}^n (a_j^2-a_k^2)} .
\end{equation}
\end{proposition}

For $t \ge \frac 1 {\sqrt 2}$, there is only one non-zero term in the sum, and this formula was used by Liu and Tkocz \cite{LT}. For $\frac 1 {\sqrt 3} \le t < \frac 1 {\sqrt 2}$ two non-zero terms are possible, formally with an alternating behavior and a singularity. This results in difficulties determining extrema of $A_1(a,t)$. Brzezinski \cite{Br} proved Proposition \ref{prop1} in detail using the Fourier transform method. Since Proposition \ref{prop1} is essential for our proof of Theorem \ref{th1}, we now outline Brzezinski's main arguments. The first step is

\begin{lemma}\label{lemma1}
Let $n, l \in \N$ with $0 \le l \le n-2$. Then for all $x \in \R^n$ with $x_1 > x_2 > \cdots > x_n >0$
\begin{equation}\label{eq2.2}
\sum_{j=1}^n \frac{x_j^{2l}}{\prod_{k=1, k \ne j}^n (x_j^2-x_k^2)} =0 \; , \; \sum_{j=1}^n \frac{x_j^{2(n-1)}}{\prod_{k=1, k \ne j}^n (x_j^2-x_k^2)} =1 \ .
\end{equation}
\end{lemma}

\begin{proof}
Let $y_{2k} :=(x_1^{2k}, \cdots , x_n^{2k}) \in \R^n$ and $y_{2k}^{(j)} :=(x_1^{2k}, \cdots , \widehat{x_j^{2k}}, \cdots x_n^{2k}) \in \R^{n-1}$, where the hat indicates omitting the $j$-th coordinate, be column vectors. Consider the matrices
$V = (y_0,y_2, \cdots , y_{2(n-1)})_{n \times n}$, $V^{(j)}= (y_0^{(j)}, \cdots , y_{2(n-2)}^{(j)})_{(n-1) \times (n-1)}$ and  \\
$W^{(l)} = (y_0, \cdots , y_{2l}, y_{2l}, \cdots , y_{2(n-2)})_{n \times n}$ : in $W^{(l)}$ column $y_{2l}$ is repeated and column $y_{2(n-1)}$ is omitted. Thus $det(W^{(l)}) = 0$. The Vandermonde determinant of $V$ is $det(V) = \prod_{1 \le k< j \le n} (x_j^2-x_k^2)$. Expanding the determinant of $W^{(l)}$ along its $j$-th column yields
$$0 = \frac{det(W^{(l)})}{det(V)} = \frac 1 {det(V)} \sum_{j=1}^n (-1)^l x_j^{2l} det(V^{(j)}) = \sum_{j=1}^n \frac{x_j^{2l}}{\prod_{k=1, k\ne j}^n (x_j^2-x_k^2)} \ . $$
The second equation follows similarly by expanding $det(V)$ along its last column.
\end{proof}

In Meyer, Pajor \cite{MP} there is a remark that expressions of the above type are related to Vandermonde determinants.

\vspace{0.5cm}

{\it Proof of Proposition \ref{prop1}.} \\
Brezezinski \cite{Br} calculates the Fourier transform of the $l_1^n$-ball
\begin{align}\label{eq2.3}
\hat{\chi}_{B_1^n}(x) &= \frac 1 {(2 \pi)^{n/2}} \int_{-1}^1 e^{-ix_1 y_1} \int_{-(1-|y_1|)}^{1-|y_1|} e^{-ix_2y_2} \cdots \int_{-(1-\sum_{j=1}^{n-1} |y_j|)}^{1-\sum_{j=1}^{n-1} |y_j|} e^{-ix_ny_n} dy_n \cdots dy_2 dy_1 \nonumber \\
&= (\frac 2 \pi)^{n/2} (-1)^{[n/2]} \sum_{j=1}^n \frac {x_j^{n-2} f(x_j)}{\prod_{k=1, k \ne j}^n (x_j^2-x_k^2)} \ ,
\end{align}
with $f(x)= \cos(x)$ if $n$ is even and $f(x)=\sin(x)$ if $n$ is odd. The last equality is shown by induction on $n$, starting with $n=2$, and using \eqref{eq2.2} in the calculation. Then $\hat{A_1}(a,s) = (2 \pi)^{(n-1)/2} \hat{\chi}_{B_1^n}(s a)$. By \eqref{eq2.3}, this function seems to have a singularity of order $\frac{s^{n-2}}{s^{2(n-1)}} = \frac 1 {s^n}$ at $s=0$. However, the singularity is removable: Again using Lemma \ref{lemma1}, one shows for all $l \le n-1$ if $n$ is even
$$\frac {d^l}{ds^l} \Big(\sum_{j=1}^n \frac{a_j^{n-2} \cos(a_j s)}{\prod_{k=1, k\ne j}^n (a_j^2-a_k^2)} \Big)_\Big{|s=0} = 0 \ , $$
whereas for $l = n-1$ this is equal to $(-1)^{[n/2]}$. If $n$ is odd, one has to replace $\cos(a_js)$ here by $\sin(a_j s)$. L'Hospital's rule then shows
$$\lim_{s \to 0} \hat{A_1}(a,s) = \frac 1 {\sqrt {2 \pi}} \frac{2^n}{n!} (-1)^m \ , $$
with $m=n$ if $n$ is even and $m=n-1$ if $n$ is odd. Then
\begin{align*}
A_1(a,t) &= \frac 1 {\sqrt{2 \pi}} \int_\R \hat{A_1}(a,s) e^{i s t} dt \\
&= \RR \Big( (-1)^{n/2} \frac {2^{n-2}}{\pi} \int_\R \sum_{j=1}^n \frac {a_j^{n-2}}{\prod_{k=1, k \ne j}^n (a_j^2-a_k^2)} \frac {e^{i(a_j+t)s} + sgn(a_j-t)^n e^{i |a_j-t|s}}{s^n} ds \Big) \ ,
\end{align*}
where $\RR$ denotes the real part if $n$ is even and the imaginary part if $n$ is odd. The function in the integrand has a pole of order $n$ at zero, and the integral can be evaluated using the residue theorem. The result is
$$A_1(a,t) = \frac {2^{n-1}}{(n-1)!} \sum_{j=1}^n \frac {a_j^{n-2} (a_j-t)_+^{n-1}} {\prod_{k=1, k \ne j}^n (a_j^2-a_k^2)} \ . $$
\hfill $\Box$

\vspace{0,5cm}

For $n=2$, Proposition \ref{prop1} yields that $A_1(a,t) = \frac 2 {a_1+a_2}$ if $a_1 \ge a_2 > t$ and $A_1(a,t) = 2 \frac {a_1-t}{a_1^2-a_2^2}$ if $a_1 \ge t >a_2$. This implies with $a^{[t]}:=(\sqrt{1-t^2},t)$ that the extremal sections of $B_1^2$ at distance $t$ are given by
$$ Maximum :  \left\{\begin{array}{c@{\quad}l}
a^{[t]} \ , &   0 \le t \le \frac 1 {{\sqrt 2}} \\
a^{(1)} \ , &   \frac 3 4 \le t < 1
\end{array}\right\}  \; , \;
Minimum : \left\{\begin{array}{c@{\quad}l}
a^{(2)} \ , &   0 \le t \le {\sqrt 2} - 1 \\
a^{(1)} \ , &   {\sqrt 2} -1 \le t < \frac 1 {{\sqrt 2}}
\end{array}\right\}  \ . $$
For $t \in (\frac 1 {\sqrt 2}, \frac 3 4)$, the maximal vector $a_{max}$ rotates from $a^{[1/\sqrt 2]}=a^{(2)}$ to $a^{(1)}$ continuously. But there is a discontinuity of $A_1(a_{max},t)$ at $t= \frac 1 {\sqrt 2}$, cf. Liu, Tkocz \cite{LT}. \\

For $n=3$, Brzezinski \cite{Br} used Proposition \ref{prop1} and tedious calculations to determine the extremal sections of $B_1^n$ at distance $t$ from the origin. His result is

\begin{proposition}\label{prop2}
For $n=3$ and $0 \le t \le 1$, let $a^{[t]} := (\sqrt{1-2 t^2},t,t)$. Then $A_1(.,t)$ has the following extrema
$$ Maximum : \left\{\begin{array}{c@{\;}l}
a^{[t]} , &  0 \le t \le \frac 1 {{\sqrt 6}} \\
a^{(3)} , &  0.4489 \le t < \frac 1 {{\sqrt 3}} \\
a^{(2)} , &  0.5826 \le t \le 0.6384 \\
a^{(1)} , &  0.6384 \le t < 1
\end{array}\right\} ,
Minimum : \left\{\begin{array}{c@{\;}l}
a^{(3)} , &  0 \le t \le 0.2125 \\
a^{(1)} , &  0.2125 \le t < \frac 1 {{\sqrt 3}}
\end{array}\right\} . $$
For $t \in [\frac 1 {\sqrt 6}, 0.4489]$, there is a continuous rotation of the maximal vector $a_{max}$ from $a^{[1/\sqrt 6]}$ to $a^{(3)}$ and for $t \in [\frac 1 {\sqrt 3}, 0.5826]$ from $a^{(3)}$ to $a^{(2)}$. Still, there is a discontinuity of $A_1(a_{max},t)$ at $t = \frac 1 {\sqrt 3}$. \\
One has $A_1(a^{(1)},t)= 2 (1-t)^2$, $A_1(a^{(2)},t)= \sqrt 2 (1-2 t^2)$, $A_1(a^{(3)},t)= \frac{3 \sqrt 3} 4 (1-t^2)$.
\end{proposition}

For $t = \sqrt 6 - \sqrt{9/2}$, $a^{(2)}$ gives an equilateral, but irregular hexagonal section. For $t=0$, the hyperplane perpendicular to $a^{(3)}$ intersects $B_1^3$ in a regular hexagon. \\

Note that $A_1(a^{(n)},\frac 1 {\sqrt n})$ is the volume of a boundary $(n-1)$-simplex, i.e.
$$A_1(a^{(n)},\frac 1 {\sqrt n}) = \frac{{\sqrt n}}{(n-1)!} =  \frac{2^{n-1}}{(n-1)!} \ \frac{{\sqrt n}}{2^{n-1}} \ . $$
We have
$$A_1(a^{(1)},\frac 1 {\sqrt n}) = \frac{2^{n-1}}{(n-1)!} (1- \frac 1 {{\sqrt n}})^{n-1} > A_1(a^{(n)},\frac 1 {\sqrt n}) $$
if and only if $n \ge 7$. Hence, different from the cases $n=2, 3$, $a^{(n)}$ does not yield the maximal hyperplane at distance $t= \frac 1 {\sqrt n}$, if $n \ge 7$. \\

Suppose that $ 1 > a_1 > a_2 > \cdots > a_{k-1} >t > a_k, \cdots a_n \ge 0$. Then by continuity, formula \eqref{eq2.1} also holds if some or all of the $a_j$ with $j \ge k$ are equal or are zero. This can be used to determine $A_1(a^{(k)},t)$ for $0 < t < \frac 1 {\sqrt k}$.

\begin{proposition}\label{prop3}
Let $n, k \in \N$ with $1 \le k \le n$ and $0 \le t < \frac 1 {\sqrt k}$. Then
\begin{equation}\label{eq2.4}
A_1(a^{(k)},t) = \frac{2^{n-1}}{(n-1)!} \frac 1 {(k-1)!} \frac {d^{k-1}}{dx^{k-1}} \Big( \frac {(\sqrt{x}-t)^{n-1}}{x^{n/2-k+1}} \Big)_\Big{|x=\frac 1 k} \ .
\end{equation}
In particular, for $k=1, 2, 3$, $A_1(a^{(k)},t)$ equals $\frac{2^{n-1}}{(n-1)!} $ times \\
$ \quad k=1 : (1-t)^{n-1} \quad , \quad  t < 1 $ \\
$ \quad k=2 : \frac 1 {\sqrt 2} (1 + {\sqrt 2} (n-2) t) (1- {\sqrt 2} \ t)^{n-2} \quad , \quad  t < \frac 1 {\sqrt 2} $ \\
$ \quad k=3 : \frac{3 {\sqrt 3}} 8 ( 1 + {\sqrt 3} (n-3) t + (n-2)(n-4) t^2) (1- {\sqrt 3} \ t)^{n-3} \quad , \quad t < \frac 1 {\sqrt 3} \ .$ \\
In general,
$$A_1(a^{(k)},t) =  \frac{2^{n-1}}{(n-1)!} \ b_k \ p_{k-1}(n,t) (1 - {\sqrt k} \ t)^{n-k} \quad , \quad t < \frac 1 {\sqrt k} \ , $$
where $p_{k-1}$ is a $(k-1)$-st order polynomial in $t$ with coefficients depending on $n$, which is starting with $1$ and $b_k = \binom{2k-2}{k-1} \frac{{\sqrt k}}{4^{k-1}}$; $b_k$ is decreasing to $\frac 1 {\sqrt \pi}$ as $k \to \infty$.
\end{proposition}

\begin{proof}
The case $k=1$ is clear from \eqref{eq2.1}. If $a_1 > a_2 >t$ satisfy $a_1^2+a_2^2 = 1$, we have by continuous extension of \eqref{eq2.1} for $a=(a_1,a_2,0, \cdots , 0)$
$$\frac{(n-1)!}{2^{n-1}} A_1(a,t) = \frac 1 {a_1^2-a_2^2} \Big( \frac{(a_1-t)^{n-1}}{a_1^{n-2}} - \frac{(a_2-t)^{n-1}}{a_2^{n-2}} \Big) \ . $$
For $k=2$, choose $a_{\varepsilon} := (\sqrt{\frac{1+\varepsilon}{2}},\sqrt{\frac{1-\varepsilon}{2}},0, \cdots ,0)$.
Then $a_{\varepsilon} \to a^{(2)}$ and $A_1(a_{\varepsilon},t) \to A_1(a^{(2)},t)$ as $\varepsilon \to 0$, if $t < \frac 1 {\sqrt 2}$ and
\begin{align*}
\frac{(n-1)!}{2^{n-1}} & A_1(a^{(2)},t) = \lim_{\varepsilon \to 0} \Big( \frac{(\sqrt{(1+\varepsilon)/2}-t)^{n-1}}{(\sqrt{(1+\varepsilon)/2})^{ n-2}} -  \frac{(\sqrt{(1-\varepsilon)/2}-t)^{n-1}}{(\sqrt{(1-\varepsilon)/2})^{n-2}} \Big) \Big/ \varepsilon \\
&= \frac d {dx}\Big (\frac{(\sqrt{x}-t)^{n-1}}{x^{n/2-1}} \Big)_\Big{|x=\frac 1 2} = \frac 1 {\sqrt 2} (1 + {\sqrt 2} (n-2) t) (1-{\sqrt 2} \ t)^{n-2} \ .
\end{align*}
For $a^{(k)}$ consider similarly $k$ equidistant points of step-size $\varepsilon$ centered around $\sqrt{\frac 1 k}$ and $n-k$ zero coordinates, yielding $k$ non-zero terms in \eqref{eq2.1}. Then the continuous extension of \eqref{eq2.1} yields that $A_1(a^{(k)},t)$ is the limit of $\frac 1 {(k-1)!}$ times the $(k-1)$-st order $\varepsilon$-size differences of the function $f(x) := \frac{({\sqrt x} -t)^{n-1}}{x^{n/2-k+1}}$ divided by $\varepsilon^{k-1}$ at $x = \frac 1 k$ as $\varepsilon \to 0$, if $t < \frac 1 {{\sqrt k}}$,  i.e.
$$ \frac{(n-1)!}{2^{n-1}} A_1(a^{(k)},t) = \frac 1 {(k-1)!} \frac {d^{k-1}} {dx^{k-1}} \Big (\frac{(\sqrt{x}-t)^{n-1}}{x^{n/2-k+1}} \Big)_\Big{|x=\frac 1 k} \ . $$
Since $a^{(k)}$ has $(n-k)$ zeros, the power in the denominator of $f$ is $(n-k)-(n-2)/2 = n/2-k+1$.
The general form of $A_1(a^{(k)},t)$ follows from \eqref{eq2.4} and the product and the quotient rule. The formula for $b_k$ comes from by Meyer, Pajor \cite{MP} for $t=0$.
\end{proof}

\section{Auxiliary results}

We now outline the basic steps of the proof of Theorem 1. Let $\frac 1 {\sqrt 3} \le t < \frac 1 {\sqrt 2}$ and $a = (a_j)_{j=1}^n \in S^{n-1}$ be given with $1 > a_1 > a_2 > a_3, \cdots , a_n \ge 0$. Two basic cases are possible: $a_1 > t \ge a_2$ and $a_1 > a_2 > t$. The first case is easier, since then there is only one term in formula \eqref{eq2.1} for the parallel section function. Here we distinguish two cases: $a_1 \le \frac 1 {\sqrt 2}$ and $a_1 > \frac 1 {\sqrt 2}$. In the first subcase, we show that $A_1(a,t) \le A_1(a^{(2)},t) < A_1(a^{(1)},t)$ for $n \ge 4$. In the second subcase, we prove that $A_1(a,t) < A_1(a^{(1)},t)$. In this case, $A_1(a,t)$ may be bigger than $A_1(a^{(2)},t)$. \\

The second case is more difficult: If $a_1 > a_2 > t$, formula \eqref{eq2.1} has two non-zero terms, alternating and with common denominator $\frac 1 {a_1^2-a_2^2}$. This formal (but removable) singularity as $a_2 \to a_1$ has to be overcome. First fixing $1 > a_1 > a_2$, we show that a maximum of $A_1(a,t)$ relative to this constraint requires that there are at most two different non-zero values among the coordinates $a_3, \cdots , a_n$. Then this is reduced to one non-zero value and only of multiplicity one in the coordinate sequence. Then we have to consider the vector $a = (a_1,a_2,\sqrt{1-a_1^2-a_2^2},0, \cdots , 0)$. Now fixing only $a_1$, we show that $A_1(a,t)$ is monotonically increasing in $a_2$, reducing $\sqrt{1-a_1^2-a_2^2}$. If $a_1 \le \frac 1 {\sqrt 2}$, we take the limit of $a_2 \to a_1$ using l'Hospital's rule and then $a_1 \to \frac 1 {\sqrt 2}$ yields a bound by $A_1(a^{(2)},t)$. If $a_1 > \frac 1 {\sqrt 2}$, we take $a_2 = \sqrt{1-a_1^2}$ and prove a bound by $A_1(a^{(1)},t)$. \\

In this section, we prove two monotonicity results for functions occurring in the second case $\frac 1 {\sqrt 3} \le t < a_2 < a_1$. Clearly $a_j < \frac 1 {\sqrt 3}$ for all $j \ge 3$. Then the proof of Theorem \ref{th1} will be reduced to the case that there is only one non-zero value among the coordinates $a_j, j \ge 3$, but may be of multiplicity $r$. Then with $c:=1-a_1^2-a_2^2$ formula \eqref{eq2.1} states that $A_1(a,t)$ is, up to a constant factor, equal to
$$\frac 1 {a_1^2-a_2^2} \Big( \frac{(a_1-t)^{n-1}}{a_1^{n-2}} (\frac {a_1^2}{a_1^2 - c/r})^r - \frac{(a_2-t)^{n-1}}{a_2^{n-2}} (\frac {a_2^2}{a_2^2 - c/r})^r \Big) \ . $$
We now show that this is a decreasing function of $r \ge 1$. For simplicity of notation we put $a = a_1$ and $b = a_2$. \\

\begin{proposition}\label{prop3.1}
Let $n \ge 4$, $r \ge 1$, $1 > a > b > t \ge \frac 1 {\sqrt 3}$ and $c := 1-a^2-b^2 >0$. Then the function $f : [1, \infty) \to \R$,
$$f(r) := \Big( \frac{(a-t)^{n-1}}{a^{n-2}} (\frac {a^2}{a^2 - c/r})^r -  \frac{(b-t)^{n-1}}{b^{n-2}} (\frac {b^2}{b^2 - c/r})^r \Big) $$
is decreasing for $r \ge 1$.
\end{proposition}

Note that $b^2-c = a^2 + 2 b^2 -1 >0$, $b^2 - c/r >0$, $a^2 - c/r >0$ for $r \ge 1$.

\begin{proof}
Let $g_a(r) := ( \frac {a^2}{a^2 - c/r} )^r$. Then
$$g_a'(r) = -g_a(r) ( \frac {c}{a^2 r - c} - \ln(1+ \frac c {a^2 r -c}) ) < 0 \ , $$
since $\ln (1+x) \le x$. Let $F :=  \frac{(a-t)^{n-1}}{a^{n-2}} g_a(r)$ and $G := \frac{(b-t)^{n-1}}{b^{n-2}} g_b(r)$. Then
$$f'(r) = - F ( \frac {c}{a^2 r - c} - \ln(1+ \frac c {a^2 r -c}) ) + G ( \frac {c}{b^2 r - c} - \ln(1+ \frac c {b^2 r -c}) ) \ . $$
This is $\le 0$ if and only if
\begin{equation}\label{eq3.1}
\frac F G  \ge  \frac {\frac {c}{b^2 r - c} - \ln(1+ \frac c {b^2 r -c})}{\frac {c}{a^2 r - c} - \ln(1+ \frac c {a^2 r -c})} \ .
\end{equation}
Now $\phi(x) := \frac{x-\ln(1+x)}{x^2}$ is decreasing in $x>0$, since $\phi'(x) = \frac 1 {x^3} (2 \ln(1+x)-\frac {x(2+x)}{1+x})$ and
$\psi(x) := 2 \ln(1+x) -\frac{x(2+x)}{1+x}$ satisfies $\psi'(x) = - \frac{x^2}{(1+x)^2} < 0$ with $\psi(0) = 0$. Thus $\phi$ is decreasing in $x>0$. Applying this to $x = \frac c {b^2 r-c}$ and $y = \frac c {a^2 r-c}$, $x > y$, shows that to prove \eqref{eq3.1} it suffices to show that
$$\frac F G \ge \Big( \frac{a^2-c/r}{b^2-c/r} \Big)^2 \quad , \text{i.e.} \ $$
\begin{equation}\label{eq3.2}
( \frac{a-t}{b-t} )^{n-1} (\frac b a )^{n-2} \ge ( \frac {b^2}{a^2} )^r ( \frac {a^2-c/r}{b^2-c/r} )^{r+2} \ .
\end{equation}
However, $h(r) := ( \frac {b^2}{a^2} )^r ( \frac {a^2-c/r}{b^2-c/r} )^{r+2}$ is decreasing in $r$, since with
$$\frac {b^2}{a^2} \frac {a^2-c/r}{b^2-c/r} = 1 + \frac{(a^2-b^2)c/r}{a^2(b^2-c/r)} $$
we have
\begin{align*}
(\ln h)'(r) &= \ln(1 + \frac{(a^2-b^2)c/r}{a^2(b^2-c/r)}) - \frac{(r+2)(a^2-b^2)c/r}{(a^2 r -c)(b^2 -c/r)} \\
&\le \frac{(a^2-b^2)c/r}{b^2-c/r}) ( \frac 1 {a^2} - \frac{r+2}{a^2 r - c} ) < 0 \ ,
\end{align*}
using $\ln(1+x)<x$ and $a^2 (r+2)> a^2 r-c$. Hence we only have to prove \eqref{eq3.2} for $r=1$, i.e.
\begin{equation}\label{eq3.3}
( \frac{a-t}{b-t} )^{n-1} (\frac b a )^{n-4} \ge  ( \frac {a^2-c}{b^2-c} )^3 \ .
\end{equation}
Since $\frac{(a-t) b}{(b-t)a} > 1$ and $n \ge 4$, the left side is bigger than $(\frac{a-t}{b-t} )^3$. Thus it suffices to show that
$(\frac{a-t}{b-t} ) \ge (\frac{a^2-c}{b^2-c} )$. This is equivalent to $(a-b) ((a+b) t - a b -c) \ge 0$. Inserting the value of $c = 1-a^2-b^2$, we find
\begin{align*}
(a+b) t - a b -c &= (a-b)^2 +(a+b) t + a b -1 \\
& \ge \frac{a+b}{{\sqrt 3}} + a b -1 = (b - \frac 1 {{\sqrt 3}})(b + \sqrt 3) + (b + \frac 1 {{\sqrt 3}})(a-b) > 0 \
\end{align*}
for all $a > b \ge t \ge \frac 1 {{\sqrt 3}}$. This implies \eqref{eq3.3} and ends the proof of Proposition \ref{prop3.1}.
\end{proof}

{\it Remark}. Similarly as above one shows that $(\frac {b^2}{a^2} )^{r-1} ( \frac {a^2-c/r}{b^2-c/r} )^{r+1}$ is decreasing in $r$. \\

\vspace{0,5cm}

For $r=1$, we are left in $A_1(a,t)$ essentially with
$$\frac 1 {a^2-b^2} \Big( \frac{(a-t)^{n-1}}{(2 a^2 + b^2 -1) a^{n-4}} - \frac{(b-t)^{n-1}}{(a^2 + 2 b^2 -1) b^{n-4}} \Big) \ , \ a > b \ge t \ge \frac 1 {{\sqrt 3}} \ . $$
To determine the maximum of this expression, we show that, fixing $a$, it is an increasing function of $b$, reducing $c = 1-a^2-b^2 >0$.

\begin{proposition}\label{prop3.2}
Let $a > b \ge t \ge \frac 1 {{\sqrt 3}}$, $b \le \frac 1 {{\sqrt 2}}$, $1-a^2-b^2 >0$ and $n \ge 4$. Fix $a, t$ and $n$. Then $f:[\frac 1 {{\sqrt 3}},\min(a,\sqrt{1-a^2})) \to \R$,
$$f(b) := \frac 1 {a^2-b^2} \Big( \frac{(a-t)^{n-1}}{(2 a^2 + b^2 -1) a^{n-4}} - \frac{(b-t)^{n-1}}{(a^2 + 2 b^2 -1) b^{n-4}} \Big) $$
is an increasing function of $b$.
\end{proposition}

\begin{proof}
Let $F(b) := \frac{(a-t)^{n-1}}{(2 a^2 + b^2 -1) a^{n-4}}$ and $G(b) := \frac{(b-t)^{n-1}}{(a^2 + 2 b^2 -1) b^{n-4}}$. Then $f(b) = \frac{F(b)-G(b)}{a^2-b^2}$ and
$$f'(b) = \frac {2b}{(a^2-b^2)^2} (F(b)-G(b)) + \frac 1 {a^2-b^2} ( F(b) (\ln F)'(b) -G(b) (\ln G)'(b) ) \ , $$
$$(a^2-b^2)^2 f'(b) = F(b) (2 b + (a^2-b^2) (\ln F)'(b)) - G(b) (2 b + (a^2-b^2) (\ln G)'(b)) \ . $$
Now $2 b + (a^2-b^2)(\ln F)'(b) = 2 b \frac{a^2 + 2 b^2 - 1}{2 a^2 + b^2 -1}$ and $(\ln G)'(b) = \frac{n-1}{b-t} - \frac{n-4} b - \frac{4 b}{a^2 + 2 b^2 - 1}$.
Therefore
$$\frac{(a^2-b^2)^2 f'(b)}{2 b G(b)} = ( \frac{a-t}{b-t} )^{n-1} ( \frac b a )^{n-4} ( \frac{a^2+ 2 b^2 -1}{2 a^2 + b^2 -1} )^2 - (1 + \frac{a^2-b^2}{2 b} (\ln G)'(b) ) \ . $$
Hence we have to show that
\begin{align*}
L & := \Big(\frac{a-t}{b-t} \Big)^{n-1} \Big( \frac b a \Big)^{n-4} \\
& \ge R :=\Big(\frac{2 a^2 + b^2 -1}{a^2+2 b^2 -1} \Big)^2 \Big(1+\frac{a^2-b^2}{2 b} (\frac{n-1}{b-t} - \frac{n-4} b - \frac{4 b}{a^2 + 2 b^2 - 1}) \Big) \ .
\end{align*}
We first estimate the left side $L$ from below using
$$ \frac{b (a-t)}{a(b-t)} = 1 + \frac t a \frac{a-b}{b-t} \ , \ \frac a b = 1 + \frac{a-b} b \ , \ (\frac a b)^2 \ge 1 + 2 \frac{a-b} b \ :$$
\begin{align*}
L &=  (1 + \frac t a \frac{a-b}{b-t})^{n-1} (\frac a b) (\frac a b)^2  \\
 & \ge (1 + (n-1) \frac t b \frac{a-b}{b-t} + \frac{a-b} b)(1 + 2 \frac{a-b} b) + \frac{(n-1)(n-2)} 2 (\frac t b)^2 (1+ \frac{a-b} b) (\frac{a-b}{b-t})^2 \\
 & + \frac{(n-1)(n-2)(n-3)} 6 (\frac t b)^3 (\frac{a-b}{b-t})^3  \\
& = 1 + L_1 (a-b) + L_2 (a-b)^2 + L_3 (a-b)^3 \ ,
\end{align*}
$L_1 = (n-1) \frac t b \frac 1 {b-t} + \frac 3 b$, $L_2 = 2 (n-1) \frac t {b^2} \frac 1 {b-t} + \frac 2 {b^2} + \frac{(n-1)(n-2)} 2 (\frac t b)^2 \frac 1 {(b-t)^2}$, \\
$L_3 = \frac{(n-1)(n-2)} 2 (\frac{t^2}{b^3}) \frac 1 {(b-t)^2} + \frac{(n-1)(n-2)(n-3)} 6 (\frac t b)^3 \frac 1 {(b-t)^3} $.
To estimate the right side $R$ from above, we note that
\begin{equation}\label{eq3.4}
(\frac{2 a^2 + b^2 - 1}{a^2 + 2 b^2 -1})^2 \le 1 + \frac {4b}{3 b^2 -1} (a-b) \ ,
\end{equation}
since
\begin{align*}
& 1+ \frac{4 b}{3 b^2 -1} (a-b) - (\frac{2 a^2 + b^2 - 1}{a^2 + 2 b^2 -1})^2 = \frac{(a-b)^2}{(3 b^2 -1)(a^2+2 b^2 -1)^2} \times \\
 & \quad \times [2(1-b^2)(3 b^2-1)+4 b (1+b^2)(a-b)+(7 b^2+3)(a-b)^2+4 b (a-b)^3] \ ,
\end{align*}
which is positive. Further,
$$\frac{a^2-b^2}{2 b} = a-b + \frac{(a-b)^2}{2 b} \ , \ \frac{n-1}{b-t} - \frac{n-4} b = (n-1)\frac t b \frac 1 {b-t} + \frac 3 b \ , $$
$$\frac{a^2-b^2}{2 b} \frac{4 b}{a^2+ 2 b^2 -1} = \frac{2 (a^2-b^2)}{a^2+2 b^2 -1} \ge \frac{4 b}{3 b^2 -1} (a-b) - \frac 3 {(3 b^2 -1)^2} (a-b)^2 \ , $$
since in the last inequality the difference of the left and right side equals
$$\frac{(a-b)^2}{(a^2+2 b^2 -1)(3 b^2-1)} [(3 b^2 -1)(1-2 b^2) + 2 b (5 - 6 b^2)(a-b) + 3 (a-b)^2 ] \ge 0 \ . $$
This implies that
\begin{align*}
R & \le (1 + \frac{4b}{3 b^2 -1}(a-b)) \Big( 1 + (n-1) \frac t {b(b-t)} (a-b) + \frac 3 b (a-b) \\
& + (n-1) \frac t {2 b^2 (b-t)} (a-b)^2 +\frac 3 {2 b^2} (a-b)^2 -\frac {4 b}{3 b^2 -1} (a-b) + \frac 3 {(3 b^2 -1 )^2} (a-b)^2 \Big) \\
& = 1 + R_1 (a-b) + R_2 (a-b)^2 + R_3 (a-b)^3 \ ,
\end{align*}
$R_1 = (n-1) \frac t b \frac 1 {b-t} + \frac 3 b$, $R_2 = (n-1) \frac t {2 b^2} \frac 1 {b-t} + \frac {12}{3 b^2 -1} + \frac 3 {2 b^2} + (n-1) \frac{4 t}{(b-t)(3 b^2 -1)} - \frac{16b^2-3}{(3 b^2 -1)^2}$, \\
$R_3 = (n-1) \frac{2t}{b (b-t) (3 b^2 -1)} + \frac 6 {b (3 b^2 -1)} + \frac{12 b}{(3 b^2 -1)^3}$.  Since
$$\frac 1 {3 b^2 -1} = \frac 1 {3(b^2-t^2) + (3 t^2 -1)} \le \frac 1 {3 (b-t) (b+t)} \le \frac 1 {2 {\sqrt 3} (b-t)} $$
and $2 \sqrt 3 \le (n-1) \frac t {b^2}$ for $n \ge 4$,
$$R_2 \le \frac 3 2 (n-1) \frac t {b^2} \frac 1 {b-t} + \frac 3 {2 b^2} + \frac{2 (n-1) t}{{\sqrt 3}} \frac 1 {(b-t)^2} \ , $$
$$R_3 \le \frac{(n-1)t + 3 (b-t)}{{\sqrt 3} b (b-t)^2} + \frac b {2 {\sqrt 3} (b-t)^3} \le \frac{n t}{{\sqrt 3} b (b-t)^2} + \frac b {2 {\sqrt 3} (b-t)^3} \ .$$
Using $\frac 1 {{\sqrt 3}} \le t \le b \le \frac 1 {{\sqrt 2}}$ and $n \ge 4$, simple estimates then show that $L_2 \ge R_2$ and $L_3 \ge R_3$. Moreover, $L_1 = R_1$. Hence $L \ge R$, which ends the proof of Proposition \ref{prop3.2}.
\end{proof}

\begin{corollary}\label{cor1}
Let $n \ge 2$, $\frac 1 {{\sqrt 3}} \le t \le b < a < 1$ with $a^2+b^2<1$. Then
$$(\frac{a-t}{b-t})  \ge ( \frac{2 a^2+ b^2 -1}{a^2 + 2 b^2 -1} )^2 \ . $$
\end{corollary}

\begin{proof}
By \eqref{eq3.4}, using $b \le \frac 1 {{\sqrt 2}}$ and $3 b^2 -1 \ge 2 {\sqrt 3} (b-t)$, we find
$$( \frac{2 a^2+ b^2 -1}{a^2 + 2 b^2 -1} )^2 \le 1 + \frac{4 b}{3 b^2 -1} (a-b) \le 1 + \frac{2 {\sqrt 2} (a-b)}{2 {{\sqrt 3}} (b-t)} \le 1 + \frac{a-b}{b-t} = \frac{a-t}{b-t} \ . $$
\end{proof}

\section{Proof of Theorem 1}

\begin{proof}

Liu and Tkocz \cite{LT} proved Theorem 1 when $\frac 1 {{\sqrt 2}} \le t < 1$. We will therefore assume that $\frac 1 {{\sqrt 3}} \le t < \frac 1 {{\sqrt 2}}$. Let $1 > a_1 > a_2 > a_3 > \cdots > a_n >0$. Then $a_i >t$ can be satisfied only for $i =1$ or $2$, since $a_j^2 \le 1 - a_1^2 -a_2^2 < 1 - 2 t^2 < \frac 1 3 < t^2$ for $j \ge 3$. Also, we assume that $n \ge 4$. There are two cases to be studied, first $a_1 > t \ge a_2$ and secondly $a_1 > a_2 > t$. \\

\vspace{0,2cm}

(i) We first consider the case that $a_1 > t \ge a_2$. Then by Proposition \ref{prop1}
$$A_1(a,t) = \frac {2^{n-1}}{(n-1)!}  \frac {a_1^{n-2} (a_1-t)^{n-1}} {\prod_{i=2}^n (a_1^2-a_i^2)} =: \frac {2^{n-1}}{(n-1)!} F(a,t) \ . $$
By continuity, this formula also holds if some or all of the $a_i$ coincide or are zero, for $i \ge 2$. Fix $1 > a_1 >t$ and $n$ and consider $F( \cdot ,t)$ as a function of $(a_2, \cdots , a_n)$ only. Let $c_j := a_j^2$ for $j = 2 , \cdots , n$. Then
$$F(a,t) = \frac {a_1^{n-2} (a_1-t)^{n-1}}{G(c_j)} \ , \ G(c_j) := \prod_{j=2}^n (a_1^2-c_j) \ , \ \sum_{j=2}^n c_j = 1 - a_1^2 \ .$$
Let $\Omega := \{ (c_2, \cdots , c_n) \ | \ 0 \le c_j \le t^2 , \sum_{j=2}^n c_j = 1 - a_1^2 \ \}$. The extrema of $F$ and $G$ on $\Omega$ are either in critical points in the interior or on the boundary. We first look for critical points of $\ln G$ in the interior of $\Omega$. Using $c_n = 1-a_1^2-\sum_{j=2}^{n-1} c_j$, we have $\ln G(c_j) = \sum_{j=2}^{n-1} \ln(a_1^2-c_j) + \ln (\sum_{j=2}^{n-1} c_j + 2a_1^2 -1)$ and
$$\frac{\partial}{\partial c_l} (\ln G)(c_j) = \frac{-1}{a_1^2-c_l} + \frac 1 {\sum_{j=2}^{n-1} c_j + 2a_1^2 -1} = 0 \ , \ l=2, \cdots , n-1 \ ,$$
which implies $c_2 = \cdots = c_{n-1}$ and $a_1^2-c_2 = (n-2) c_2 + 2 a_1^2 -1$, $c_2 = c_n = \frac{1-a_1^2}{n-1}$: all $c_j$ are equal. \\
If the maximum of $F$ is attained on the boundary of $\Omega$, either some $c_j$'s are zero or some $c_j$ is equal to $t^2$. \\

(a) We first consider the case of some $c_j$'s being zero. If there are $p$ zero coordinates $c_j=0$ in a maximum point of $F$, we have $s=n-1-p$ non-zero $c_j$'s, say $c_2 , \cdots , c_{s+1}$ with $\sum_{j=2}^{s+1} c_j = 1 - a_1^2$. Repeating the above argument in this situation, we find that these $c_j$ are all equal to $c_2 = \frac{1-a_1^2}{s}$. This means with $n-2-2p = 2s - n$ that
$$F(a,t) = \frac{(a_1-t)^{n-1}}{a_1^n} (\frac{a_1^2}{a_1^2- \frac{1-a_1^2} s} )^s \ . $$
If $a_1 \le \frac 1 {{\sqrt 2}}$, $c_2=a_2^2 = (1-a_1^2)/s < a_1^2$, which requires $s \ge 2$. If $a_1 > \frac 1 {{\sqrt 2}}$, $s=1$ is possible. The function
$g(s) := (\frac{a_1^2}{a_1^2 - c/s})^s$ is decreasing in $s$, since \\
$g'(s) = g(s) (\ln(1+ \frac{c/s}{a_1^2-c/s}) - \frac{c/s}{a_1^2-c/s}) \le 0 \ , $
using $\ln(1+x) \le x$. Thus $F$ is maximal for $s=1$ or $s=2$, depending on whether $a_1 > \frac 1 {{\sqrt 2}}$ or $a_1 \le \frac 1 {{\sqrt 2}}$. First, let
$\frac 1 {{\sqrt 3}} \le t < a_1 \le \frac 1 {{\sqrt 2}}$, $s=2$ and consider
$h(a_1) := \frac{(a_1-t)^{n-1}}{a_1^{n-4}} \frac 1 {(3a_1^2-1)^2}$, now as a function of $a_1$. Then
$$(\ln h)'(a_1) = \frac 1 {a_1 (a_1-t) (3 a_1^2-1)} (t [n(3a_1^2-1)+4] - 3 a_1 (1+a_1^2)) \ . $$
This is positive, since
\begin{align*}
t [n(3a_1^2-1)+4] & - 3 a_1 (1-a_1^2) = [n(3a_1^2-1)+4] (t- \frac 1 {{\sqrt 3}}) \\
& \quad + 2(n-3)(a_1 - \frac 1 {{\sqrt 3}}) + \sqrt 3 (n-3) (a_1 - \frac 1 {{\sqrt 3}})^2 - 3 (a_1 - \frac 1 {{\sqrt 3}})^3
\end{align*}
is positive for $n \ge 4$ and $1 \ge a_1 > t \ge \frac 1 {{\sqrt 3}}$. Hence $h$ is strictly increasing, $h(a_1) < h(1)$ for $a_1 < 1$. Thus $F(a,t) \le (1-t)^{n-1}$, $A_1(a,t) \le A_1(a^{(1)},t)$, with equality only for $a_1 = 1$, i.e. for $a = a^{(1)}$. \\

If $a_1 > \frac 1 {{\sqrt 2}}$ and $s=1$, consider $k(a_1) := \frac{(a_1-t)^{n-1}}{a_1^{n-2}} \frac 1 {(2a_1^2-1)}$ as a function of $a_1$. In this case
$$(\ln k)'(a_1) = \frac 1 {a_1 (a_1-t) (2 a_1^2-1)} (t [n(2 a_1^2-1) + 2] -a_1 (1+2 a_1^2)) \ .$$
Here $\phi(a_1) := t [n(2 a_1^2-1) + 2] -a_1 (1+2 a_1^2)$ is negative at $a_1 = \frac 1 {{\sqrt 2}}$, since $\frac 1 {{\sqrt 3}} \le t < \frac 1 {{\sqrt 2}}$ and positive at $a_1=1$. We find that
\begin{align*}
\phi'(a_1) & = 4 n t a_1 -1 - 6 a_1^2 = (4 n t - 6 {\sqrt 2}) a_1 +2 - 6 (a_1-\frac 1 {{\sqrt 2}})^2 \\
& \ge (\frac {16 {\sqrt 3}} 3 - 6 {\sqrt 2}) a_1 + 2 - \frac 2 3 > 0 \ ,
\end{align*}
with $a_1 - \frac 1 {{\sqrt 2}} < \frac 1 3$, $6 (a_1-\frac 1 {{\sqrt 2}})^2 < \frac 2 3$. Hence $\phi$ is strictly increasing which means that $k$ is first decreasing near $a_1 = \frac 1 {{\sqrt 2}}$ and then increasing. Thus the maximum of $k$ is attained either for $a_1 \searrow \frac 1 {{\sqrt 2}}$ or for $a_1 =1$. Since $a_2^2 = 1- a_1^2 < t^2$,  we have $\frac{a_1-t}{2 a_1^2-1} \le \frac{a_1 - \sqrt{1-a_1^2}}{2 a_1^2 -1} = \frac 1 {a_1 + \sqrt{1-a_1^2}} \le 1$ and
$$\lim_{a_1 \searrow 1/{\sqrt 2}} k(a_1) \le \lim_{a_1 \searrow 1/{\sqrt 2}} (1- \frac t {a_1})^{n-2} = (1- {\sqrt 2} \ t)^{n-2} < (1-t)^{n-1} \ .$$
The last inequality holds since $(\frac{1-\sqrt{2/3}}{1-\sqrt{1/3}})^{n-2} < \frac 1 4 < 1 - t$ using $\frac 1 {{\sqrt 3}} \le t < \frac 1 {{\sqrt 2}}$, $n \ge 4$. Hence also in this case $A_1(a,t) \le A_1(a^{(1)},t)$. For $a \ne a^{(1)}$, there is a strict inequality since $k'>0$ near $a_1 = 1$. \\

(b) Now consider the case that some $c_j$ is equal to $t^2$, e.g. $c_2 = t^2$. Since $a_1^2+t^2 > \frac 2 3$, all other $c_j$'s satisfy
$c_j \le 1-a_1^2-t^2 < \frac 1 3 < t^2$, $j = 3 , \cdots , n$ and $\sum_{j=3}^n c_j = 1 - a_1^2 - t^2$. Repeating the above argument for these $(c_j)_{j=3}^n$, we find that the maximum in this case is obtained for $c_3=1-a_1^2-t^2$, $c_4 = \cdots = c_n =0$. Again the maximum occurs when only one coordinate is non-zero, which follows from the fact that $(\frac{a_1^2}{a_1^2+t^2+ \frac{1-a_1^2-t^2} s})^s$ is decreasing in $s$, since its $s$-derivative satisfies with $B:=1-a_1^2-t^2$
$$\ln ( 1 + \frac{B/s-t^2}{a_1^2+t^2-B/s} ) - \frac{B/s}{a_1^2+t^2-B/s} \le - \frac{t^2}{a_1^2+t^2-B/s} < 0 \ . $$
We find for $F$, with $n-3$ of the $c_j$'s being zero,
$$F(a,t) = \frac{a_1-t}{a_1^2-t^2} \frac{a_1^{n-2} (a_1-t)^{n-2}}{(a_1^2-c_3) a_1^{2(n-3)}} = \frac 1 {a_1+t} (1-\frac t {a_1})^{n-2} \frac{a_1^2}{2a_1^2+t^2-1} =: l(a_1) \ . $$
Note that $a_1^2+t^2 \le 1$, so that $a_1 \le \sqrt{1-t^2} \le \sqrt{\frac 2 3}$. The function $l$ is increasing in $a_1$, since
$(\ln l)'(a_1) = (\frac 2 {a_1} - \frac 1 {a_1+t}) + (\frac{(n-2) t}{a_1 (a_1-t)} - \frac{4 a_1}{2a_1^2+t^2-1}) \ge 0 \ : $
the difference of the first two terms is obviously non-negative and the difference of the last two terms also: the latter is equivalent to
$m(t,a_1):=t^3+ (4 a_1^2-1) t - 2 a_1^3 \ge 0$. $m$ is increasing in $t$ and thus minimal for $t = \frac 1 {\sqrt 3}$,
$m(\frac 1 {\sqrt 3},a_1) = \frac 4 {\sqrt 3} a_1^2 - 2 a_1^3 - \frac 2 {3 \sqrt 3}$. This is $0$ in $a_1=\frac 1 {\sqrt 3}$, strictly increasing in
$(\frac 1 {\sqrt 3},\frac 4 {3 \sqrt 3})$ and strictly decreasing in $(\frac 4 {3 \sqrt 3},\sqrt{\frac 2 3})$ with $m(\frac 1 {\sqrt 3},\sqrt{\frac 2 3}) \simeq 0.066 >0$. Hence $l$ and $F$ are increasing in $a_1$. Therefore $F(a,t) \le l(\sqrt{1-t^2}) = \frac 1 {t+\sqrt{1-t^2}} (1- \frac t {\sqrt{1-t^2}})^{n-2}$. We claim that this is strictly less than $(1-t)^{n-1}$, i.e. $A(a,t) < A(a^{(1)},t)$ in this case. The claim is equivalent to
$\phi_n(t) := \frac 1 {t+\sqrt{1-t^2}} \frac{ (1- \frac t {\sqrt{1-t^2}})^{n-2} }{(1-t)^{n-1}} < 1$. It is easily seen that $\phi_n$ is decreasing in $t$ and thus maximal for $t = \frac 1 {\sqrt 3}$, $\phi_n(\frac 1 {\sqrt 3}) = \frac{\sqrt 3}{1+ \sqrt 2}\frac{ (1-\frac 1 {\sqrt 2})^{n-2} }{(1-t)^{n-1} }$. The latter is maximal for $n=4$, if $n \ge 4$, and then $\phi_4(\frac 1 {\sqrt 3}) = \frac{\sqrt 3}{1+ \sqrt 2}\frac{ (1-\frac 1 {\sqrt 2})^2 }{(1-t)^3} \simeq 0.8152 < 1$.

\vspace{0,5cm}

(ii) We now consider the case that $\frac 1 {\sqrt 3} < t < a_2 < a_1 < 1$. Then $a_i < \frac 1 {\sqrt 3} < t$ for $ i \ge 3$. By Proposition \ref{prop1}
\begin{align*}
A_1(a,t) &= \frac{2^{n-1}}{(n-1)!} \Big( \frac {a_1^{n-2} (a_1-t)^{n-1}}{(a_1^2-a_2^2) \prod_{i=3}^n (a_1^2-a_i^2)} - \frac {a_2^{n-2} (a_2-t)^{n-1}}{(a_1^2-a_2^2) \prod_{i=3}^n (a_2^2-a_i^2)} \Big) \\
& =: \frac{2^{n-1}}{(n-1)!} (F(a,t) - G(a,t)) \ .
\end{align*}
By continuity, this formula also holds if some or all of the $a_i$'s, $3 \le i \le n$ are equal or are zero. Fix $n, 1 > a_1 > a_2 > t$ and let $c_i := a_i^2$, $i \ge 3$. Then $\sum_{i=3}^n c_i = 1-a_1^2-a_2^2 =: B$. Consider $A_1( \cdot ,t)$ as a function on $\Omega := \{(c_3, \cdots , c_n) \ | \ \sum_{i=3}^n c_i = B \}$. Note that for $i \ge 3$, $c_i \le 1-a_1^2-a_2^2 < 1 - 2 t^2 \le \frac 1 3 \le t^2$.
\begin{align*}
F(a,t) - G(a,t) &= \frac 1 {a_1^2-a_2^2} \Big( \frac{a_1^{n-2} (a_1-t)^{n-1}}{\prod_{i=3}^n (a_1^2-c_i)} -  \frac{a_2^{n-2} (a_2-t)^{n-1}}{\prod_{i=3}^n (a_2^2-c_i)} \Big) \\
& =: \frac{\tilde{F}(c_i) - \tilde{G}(c_i)}{a_1^2-a_2^2} \ .
\end{align*}
To find the maximum of $A_1$, we first look for critical points $(c_i)_{i=3}^n$ of the Lagrange function
$$L(c_i,\lambda) = \tilde{F}(c_i) - \tilde{G}(c_i) + \lambda ( \sum_{i=3}^n c_i -1 +a_1^2+a_2^2)$$
in the interior of $\Omega$,
$$\frac{\partial L}{\partial c_l} = \frac{\tilde{F}(c_i)}{a_1^2-c_l} - \frac{\tilde{G}(c_i)}{a_2^2-c_l} + \lambda = 0 \ , \ l = 3 , \cdots , n \ . $$
This means that the $c_l$ are solutions of the quadratic equation in $x$
\begin{equation}\label{eq4.1}
\tilde{F}(c_i)(a_2^2 - x) - \tilde{G}(c_i)(a_1^2 - x) + \lambda (a_1^2 - x)(a_2^2 - x) = 0 \ .
\end{equation}
There are (at most) two possible different solutions $ x = \beta_1, \beta_2$ for the $c_i$'s; possibly one one, if e.g. $\beta_2$ does not satisfy
$0 \le \beta_2 < t^2$. If the maximum of $A_1$ is attained on the boundary of $\Omega$, $p$ of the $c_i$'s are zero. The above argument shows that among the $n-2-p$ non-zero $c_i$'s, $i \ge 3$ there are (at most) two different values. Suppose we have two different solutions $c_3$ of multiplicity $r$ and $c_n$ of multiplicity $s$, $r+s=n-2-p$. Then $r c_3 + s c_n = B$. Let $c :=r c_3$ and $d:=s c_n$. Then $c+d=B$ and we have using $n-2-2p = 2(r+s)-(n-2)$, also if $p=0$,
$$\tilde{F}(c,d) - \tilde{G}(c,d) = \frac{(a_1-t)^{n-1} (a_1^2)^{r+s}}{a_1^{n-2} (a_1^2-c/r)^r (a_1^2-d/s)^s} - \frac{(a_2-t)^{n-1} (a_2^2)^{r+s}}{a_2^{n-2} (a_2^2-c/r)^r (a_2^2-d/s)^s} \ . $$
Fixing $t, a_2, a_1, r$ and $s$ and using $d = B-c$, consider this as a function of $c$ only, $\phi(c) := \tilde{F}(c,B-c) - \tilde{G}(c,B-c)$. Then
\begin{align}\label{eq4.2}
\phi'(c) &= \Big( \frac{\tilde{F}(c,B-c)}{(a_1^2-c/r) (a_1^2-(B-c)/s)} - \frac{\tilde{G}(c,B-c)}{(a_2^2-c/r) (a_2^2-(B-c)/s)} \Big) (\frac c r - \frac{B-c} s) \nonumber \\
& =: \psi(c) (\frac c r - \frac{B-c} s) \ .
\end{align}
We will prove that $\psi(c) >0$. Then $sgn(\phi'(c)) = sgn(\frac c r - \frac{B-c} s)$. This implies that $\phi$ is increasing if $c > B \frac r {r+s}$ and decreasing if $c < B \frac r {r+s}$, with $c \in [0,B]$. Hence the maximum of $\phi$ is attained either for $c=B$ or for $c=0$. In the first case, the maximum of $\phi$ as a function of $c$ is
\begin{equation}\label{eq4.3}
\phi(B) = \frac{(a_1-t)^{n-1}}{a_1^{n-2}} (\frac{a_1^2}{a_1^2-B/r})^r - \frac{(a_2-t)^{n-1}}{a_2^{n-2}} (\frac{a_2^2}{a_2^2-B/r})^r \ .
\end{equation}
In the second case, $r$ and $s$ are exchanged. Clearly, \eqref{eq4.3} is the case of only one relevant solution $\beta_1$ of \eqref{eq4.1}. \\

The claim $\psi(c) >0$ is equivalent to
\begin{equation}\label{eq4.4}
(\frac{a_1-t}{a_2-t})^{n-1} (\frac{a_2}{a_1})^{n-2} (\frac{a_1}{a_2})^{2(r+s)} > (\frac{a_1^2-c/r}{a_2^2-c/r})^{r+1} (\frac{a_1^2-(B-c)/s}{a_2^2-(B-c)/s})^{s+1} =: g(c) \ .
\end{equation}
Assume without loss of generality that $r \ge s$. We then prove that $g$ is decreasing or first decreasing and then increasing: Calculation shows, substituting $B = 1-a_1^2-a_2^2$,
$$(\ln g)'(c) = (a_1^2-a_2^2)/[(a_1^2 r-c) (a_1^2 s- (B-c)) (a_2^2 r - c) (a_2^2 s - (B-c))] \times M(c) \ , $$
$$M(c) = b_0 +b_1 c + b_2 c^2 \ , \ b_2= (r-s)(r+s+1) \ , $$
$$b_1 = r [(a_1^2+a_2^2) s (r+s+2) -2 (r+1) (1-a_1^2-a_2^2)] \ , $$
$$b_0 = r [ (r+1)(s+1)(a_1^2+a_2^2)^2 + (r+1) - s (r-s) a_1^2 a_2^2 -(r+1)(s+2)(a_1^2+a_2^2)] \ .$$
We find $M'(c) = 2 (r-s) (r+s+1) c + b_1 > 0$ since with $a_1 > a_2 > t \ge \frac 1 {{\sqrt 3}}$, $s \ge 1$
\begin{align*}
b_1/r &\ge (a_1^2+a_2^2)(r+3) - 2 (r+1) (1-a_1^2-a_2^2) = (3 r +5) (a_1^2+a_2^2) - 2 (r+1) \\
& \ge \frac 2 3 (3 r +5) - 2 (r+1) = \frac 4 3 > 0 \ .
\end{align*}
Therefore $M$ is strictly increasing in $c$. However, $b_0 < 0$: With $\frac 2 3 < x := a_1^2+a_2^2 < 1$ we have $b_0/r \le (r+1)(s+1)x^2-(r+1)(s+2)x+(r+1) \le 0$, since
$y^2 - \frac{s+2}{s+1} y +\frac 1 {s+1} = 0$ has the solutions $y=1$ and $y=\frac 1 {s+1}$ and $\frac 1 {s+1} < x < 1$. Hence $M(0) < 0$ and $M$ is increasing. Thus either $M$ stays negative or changes sign just once from negative to positive. This means that $\ln g$ is decreasing or first decreasing and then increasing in $[0,B]$. Therefore $g$ attains its maximum either at $0$ or at $B$,
$$g(c) \le \max \Big( (\frac{a_1}{a_2})^{2(r+1)} (\frac{a_1^2-B/s}{a_2^2-B/s})^{s+1} ,  (\frac{a_1}{a_2})^{2(s+1)} (\frac{a_1^2-B/r}{a_2^2-B/r})^{r+1} \Big) \ . $$
In the case that the maximum is attained at the second term, \eqref{eq4.4} requires
\begin{equation}\label{eq4.5}
(\frac{a_1-t}{a_2-t})^{n-1} (\frac {a_2}{a_1})^{n-2} > (\frac{a_2^2}{a_1^2})^{r-1} (\frac{a_1^2-B/r}{a_2^2-B/r})^{r+1} =: h(r) \ .
\end{equation}
If the maximum occurs at the first term, $r$ and $s$ are exchanged. By the Remark following the proof of Proposition \ref{prop3.1}, $h$ is decreasing in $r$, so the maximum on the right side of \eqref{eq4.5} is attained for $r=1$ and we have to verify \eqref{eq4.5} for $r=1$. Since $\frac{a_2 (a_1-t)}{a_1 (a_2-t)} \ge 1$, we get by Corollary \ref{cor1}
$$(\frac{a_1-t}{a_2-t})^{n-1} (\frac{a_2}{a_1})^{n-2} \ge \frac{a_1-t}{a_2-t} \ge (\frac{2 a_1^2 +a_2^2-1}{a_1^2+ 2 a_2^2 -1})^2 = (\frac{a_1^2-B}{a_2^2-B})^2 = h(1) \ , $$
i.e. \eqref{eq4.5} and \eqref{eq4.4} hold. Hence the maximum of $\phi$ is given by \eqref{eq4.3}. By Proposition \ref{prop3.1}, the function in \eqref{eq4.3} is decreasing in $r$, thus has its maximum for $r=1$, leading to
\begin{equation}\label{eq4.6}
A_1(a,t) \le \frac{2^{n-1}}{(n-1)!} \frac 1 {a_1^2-a_2^2} \Big( \frac{(a_1-t)^{n-1}}{a_1^{n-4}} \frac 1 {2 a_1^2 + a_2^2 -1} - \frac{(a_2-t)^{n-1}}{a_2^{n-4}} \frac 1 {a_1^2 + 2 a_2^2 -1} \Big) \ ,
\end{equation}
using that $B = 1-a_1^2-a_2^2$. Fixing $a_1$, we may increase $a_2$ as long as $a_2 \le \min(a_1 , \sqrt{1-a_1^2})$, so that $B \ge 0$. By Proposition \ref{prop3.2} the function in \eqref{eq4.6} is increasing in $a_2$, reducing $B>0$. The maximum of \eqref{eq4.6} with fixed $a_1$ is therefore obtained for $a_2 \to a_1$ if $a_1 \le \frac 1 {{\sqrt 2}}$ or for $a_2 = \sqrt{1-a_1^2}$ if $a_1 > \frac 1{{\sqrt 2}}$. \\

\vspace{0,5cm}

(iii) Consider first the case that $a_1 \le \frac 1 {{\sqrt 2}}$ in \eqref{eq4.6}. L'Hospital's rule yields that
\begin{align*}
H(a_1,t) & := \lim_{a_2 \to a_1} \frac 1 {a_1^2-a_2^2} \Big( \frac{(a_1-t)^{n-1}}{a_1^{n-4}} \frac 1 {2 a_1^2 + a_2^2 -1} - \frac{(a_2-t)^{n-1}}{a_2^{n-4}} \frac 1 {a_1^2 + 2 a_2^2 -1} \Big) \\
& = \frac 1 2 (1- \frac t {a_1})^{n-2} [ \frac{a_1 (7 a_1^2-3)  + 2 a_1^2 t}{(3 a_1^2 -1 )^2} + \frac {(n-4) t}{3 a_1^2-1} ] \\
& = \frac 1 6 (1- \frac t {a_1})^{n-4} (\frac{a_1-t}{a_1- \frac 1 {{\sqrt 3}}})^2  [ \frac{7 a_1^2-3+2a_1 t}{a_1 ({\sqrt 3} a_1 +1)^2} + (n-4) \frac{({\sqrt 3} a_1 -1) t}{a_1^2 ({\sqrt 3} a_1+1)} ] \ .
\end{align*}
The first two factors are clearly increasing in $a_1$, since $a_1 \ge t \ge \frac 1 {{\sqrt 3}}$ (or constant if $n=4$), and both terms in the brackets $[ \cdot ]$, too: calling the first $h_1(a_1,t)$ and the second $h_2(a_1,t)$, we have
\begin{align*}
\frac{\partial}{\partial a_1} h_1(a_1,t) & = \frac 1 {a_1^2 ({\sqrt 3} a_1+1)^3} [3 + 9 {\sqrt 3} a_1+7 a_1^2-4 {\sqrt 3} t a_1^2 - 7 {\sqrt 3} a_1^3] \\
& \ge \frac 1 {a_1^2 ({\sqrt 3} a_1+1)^3} ({\sqrt 3} a_1 +1) (3+6 {\sqrt 3} a_1 -11 a_1^2) \ ,
\end{align*}
where we used $t \le a_1$ and factored the result. This is positive for all $a_1 \in [0,1]$. For $h_2$ we have
$\frac{\partial}{\partial a_1} h_2(a_1,t) = \frac {2 (n-4) t} {a_1^3 ({\sqrt 3} a_1 +1)^2} (1+ {\sqrt 3} a_1 - 3 a_1^2)$ which also in positive in $[0, \frac 1 {{\sqrt 2}}]$. Hence $H(a_1,t)$ is increasing in $a_1 \in [\frac 1 {{\sqrt 3}}, \frac 1 {{\sqrt 2}}]$ for all $t$ in this range, with $a_1 \ge t$. Hence for $a_1 \le \frac 1 {{\sqrt 2}}$
$$A_1(a,t) \le \frac{2^{n-1}}{(n-1)!} H(\frac 1 {{\sqrt 2}},t) = \frac{2^{n-1}}{(n-1)!} \frac 1 {{\sqrt 2}} (1+{\sqrt 2} (n-2) t)(1- {\sqrt2} t)^{n-2} = A_1(a^{(2)},t) \ .$$
For $n \ge 4$ and $\frac 1 {{\sqrt 3}} \le t \le \frac 1 {{\sqrt 2}}$, $A_1(a^{(2)},t) < A_1(a^{(1)},t)$, i.e.
\begin{equation}\label{eq4.7}
\frac 1 {{\sqrt 2}} (1+{\sqrt 2} (n-2) t)(1- {\sqrt2} t)^{n-2} < (1-t)^{n-1} \ :
\end{equation}
Taking the logarithmic derivative of the quotient of the left and the right side in $t$, it is easily seen that the worst case for \eqref{eq4.7} to hold occurs for $t = \frac 1 {{\sqrt 3}}$, requiring $(\frac 1 {{\sqrt 2}} + \frac{n-2}{{\sqrt 3}})(1- \sqrt{\frac 2 3})^{n-2} < (1 - \frac 1 {{\sqrt 3}})^{n-1}$,
which is true for all $n \ge 4$, but not for $n=3$. For $n=3$, $A_1(a^{(2)},t) < A_1(a^{(1)},t)$ only holds if $t \ge 0.6384$, cf. Proposition \ref{prop2}. \\

\vspace{0,5cm}

(iv) In the second case, $a_1 > \frac 1 {{\sqrt 2}}$, $a_2 \le \sqrt{1-a_1^2}$, the maximum in \eqref{eq4.6} is attained for $a_2 = \sqrt{1-a_1^2}$, with
$2 a_1^2 + a_2^2 -1 = a_1^2$, $a_1^2+ 2 a_2^2 -1 = a_2^2$,
\begin{equation}\label{eq4.8}
A_1(a,t) \le \frac{2^{n-1}}{(n-1)!} f(a_1,t) \ , \ f(a_1,t) := \frac 1 {2 a_1^2 -1} \Big( \frac{(a_1-t)^{n-1}}{a_1^{n-2}} - \frac{(a_2-t)^{n-1}}{a_2^{n-2}} \Big) \ .
\end{equation}
We claim that in this case, too, $A_1(a,t) < A_1(a^{(1)},t)$, i.e. $f(a_1,t) < (1-t)^{n-1}$ if $a_1 < 1$. \\

To prove this, first consider the case that $n \ge 5$ and $a_1 \ge \frac{18}{25} = 0.72$. Since the second term for $f$ in \eqref{eq4.8} is negative, $f(a_1,t) \le \frac{a_1}{2 a_1^2 -1} ( 1 - \frac t {a_1})^{n-1} =: g(a_1,t)$. We will show that $g(a_1,t) < (1-t)^{n-1}$, i.e.
\begin{equation}\label{eq4.9}
h(a_1,t) := \frac{a_1}{2 a_1^2 -1} (\frac{1-\frac t {a_1}}{1-t})^{n-1} < 1.
\end{equation}
As a function of $t$, $h$ is decreasing since $\frac d {dt} (\frac{1-\frac t {a_1}}{1-t}) = - \frac{1-a_1}{a_1 (1-t)^2} < 0 $. Therefore $n=5$ and $t = \frac 1 {{\sqrt 3}}$ are the worst cases for \eqref{eq4.9} to hold, and we have to show
$$\phi(a_1) := \frac{a_1}{2 a_1^2 -1} ( \frac{1- \frac 1 {{\sqrt 3} a_1} }{1-\frac 1 {{\sqrt 3}} } )^4 =: \frac{k(a_1)}{(1-\frac 1 {{\sqrt 3}} )^4 } < 1 \ . $$
But $(\ln k)'(a_1) = - \frac{1+2 a_1^2}{a_1 (2 a_1^2 -1)} + \frac 4 {a_1 ( {\sqrt 3} a_1-1 )} < 0$, if
$4(a_1^2-1) < (1+2 a_1^2)({\sqrt 3} a_1 -1)$, i.e. $a_1 < a_0 \simeq 0.7682$. For $a_1 > a_0$, $(\ln k)'(a_1) > 0$. Thus $\ln k$ is decreasing in $[0.72,a_0]$ and increasing in $[a_0,\sqrt{\frac 2 3}]$. Note that $a_1 = \sqrt{1-a_2^2} \le \sqrt{1-t^2} \le \sqrt{\frac 2 3}$. Hence for all $a_1 \in [0.72,\sqrt{\frac 2 3}]$
$$\phi(a_1) \le \max ( \phi(0.72), \phi(\sqrt{\frac 2 3}) ) = \max (0.945 , 0.565) < 1 \ . $$

Consider next the case $n \ge 5$ and $\frac 1 {{\sqrt 2}} < a_1 <0.72$. Then $\frac 1 {a_1+a_2} < 0.71$, $a_2 = \sqrt{1-a_1^2} \ge 0.69$ and
$$f(a_1,t) = \frac 1 {a_1+a_2} (\frac{(a_1-t)^{n-1}}{a_1^{n-2}} - \frac{(a_2-t)^{n-1}}{a_2^{n-2}}) / (a_1-a_2) \le 0.71 \ g'(\theta) \ , $$
$g(x) := \frac{(x-t)^{n-1}}{x^{n-2}}$, $t < a_2 < x < a_1 \le 0.72$, for some $\theta \in [0.69,0.72]$. But with $t \le x$,
$$g'(x) = \frac{(n-2) t + x}{x^{n-1}} (x-t)^{n-2} \le (n-1) (1-\frac t x)^{n-2} \ . $$
The claim $f(a_1,t) < (1-t)^{n-1}$ will follow from $\psi(t) := 0.71 (n-1) \frac{(1-t/x)^{n-2}}{(1-t)^{n-1}} < 1$. $\psi$ is decreasing in $t$, since $\ln(\psi)'(t) = - \frac{t+n-2-(n-1) x}{(1-t)(x-t)} < 0$. Thus we need
$$\frac{0.71}{1-\frac 1 {{\sqrt 3}}} (n-1) (\frac{1-\frac 1 {{\sqrt 3} x}}{1- \frac 1 {{\sqrt 3}}})^{n-2} \le 1.7 (n-1) (\frac{1-\frac 1 {{\sqrt 3} x}}{1- \frac 1 {{\sqrt 3}}})^{n-2} < 1 \ . $$
Now $1-\frac 1 {{\sqrt 3} x}$ is maximal for $x = 0.72$ and we require $1.7 (n-1) 0.47^{n-2} < 1$. Since $(m+1) q^m$ is decreasing in $m \ge 1$ if $0<q<1$, we need $1.7 \times 4 \times 0.47^3 < 1$ which is true. \\

\vspace{0,5cm}

(v) Last, consider the case $n=4$ and $a_1 > \frac 1 {\sqrt 2}$. Then we want to show with $a_2 = \sqrt{1-a_1^2}$ that
\begin{align}\label{eq4.10}
f(a_1,a_2,t) &= \frac 1 {a_1^2-a_2^2} (\frac{(a_1-t)^3}{a_1^2} - \frac{(a_2-t)^3}{a_2^2}) \nonumber \\
&= \frac{a_1^2a_2^2 - 3 a_1a_2 t^2 + (a_1+a_2) t^3}{a_1^2 a_2^2 (a_1+a_2)} < (1-t)^3  \ .
\end{align}
Let $\phi(t) := \frac{a_1^2a_2^2 - 3 a_1a_2 t^2 + (a_1+a_2) t^3}{(1-t)^3}$. Then
$$\phi'(t) = \frac 3 {(1-t)^4} [a_1^2 a_2^2 - 2 a_1 a_2 t + (a_1 + a_2 -a_1 a_2)t^2] =: \frac 3 {(1-t)^4} \psi(a_1,a_2,t) \ , $$
$\frac{\partial}{\partial a_1} \psi(a_1,\sqrt{1-a_1^2},t) = \frac 1 {a_2}[2(2a_1^2-1)(t-a_1 \sqrt{1-a_1^2})+ t^2 (2 a_1^2-a_1-1+\sqrt{1-a_1^2}]$.
This is positive since $t > \frac 1 2 >a_1 \sqrt{1-a_1^2}$ and $\gamma(a_1) := 2 a_1^2 - a_1 -1 +\sqrt{1-a_1^2}$ is non-negative in $[\frac 1 {{\sqrt 2}}, \sqrt{\frac 2 3}]$ in view of $\gamma(\frac 1 {{\sqrt 2}})=0$ and $\gamma'(a_1) = 4 a_1 - 1 - \frac {a_1}{\sqrt{1-a_1^2}} > 0$, using $4 > \frac 1 {a_1} + \frac 1 {\sqrt{1-a_1^2}}$: At $a_1 = \sqrt{\frac 2 3}$ we have $4 > \sqrt{\frac 3 2} + \sqrt 3$. Therefore $\frac{\partial \psi}{\partial a_1} > 0$ and thus $\psi$ is maximal at $a_1 = \sqrt{1-t^2}$. Note here that $t \le a_2 = \sqrt{1-a_1^2}$ implies $a_1 \le \sqrt{1-t^2}$. We have $\psi(\sqrt{1-t^2},t,t) = t^2 \lambda(t)$,
$\lambda(t) := (1+t)(1-\sqrt{1-t^2})-t^2$. Since $(1+t-t^2)^2 -(1+t)^2 (1-t^2) = t^2 (2 t^2-1) \le 0$, $\lambda(t) \le 0$ for all $t \in [0, \frac 1 {{\sqrt 2}}]$. Therefore $\phi' < 0$ and $\phi$ is decreasing in $t$, with maximum at $t = \frac 1 {{\sqrt 3}}$. Hence \eqref{eq4.10} means that we have to prove
$$M := \frac{a_1^2 a_2^2 - a_1 a_2 + (a_1 + a_2)/(3 {\sqrt 3})}{a_1^2 a_2^2 (a_1+a_2)} < (1- \frac 1{{\sqrt 3}})^3 = 2- \frac{10} 9 {\sqrt 3} \simeq 0.0755 \ . $$
Now, $a_1^2 (1-a_1^2)(a_1+\sqrt{1-a_1^2})$ is decreasing, its minimum at $\sqrt{\frac 2 3}$ is $\frac 2 9 \frac {{\sqrt 2} +1}{{\sqrt 3}} \simeq 0.3098 \ , $
$\delta(a_1) := a_1^2 (1-a_1^2) - a_1 \sqrt{1-a_1^2} + (a_1+\sqrt{1-a_1^2})/(3 {\sqrt 3})$ is decreasing, too, since
$\delta'(a_1) = - \frac 1 {9 a_1} [ {\sqrt 3} (a_1 - \sqrt{1-a_1^2}) - 9 (a_1+ \sqrt{1-a_1^2}) (a_1 - \sqrt{1-a_1^2})^3]$. This is negative, since $a_1+\sqrt{1-a_1^2} \le {\sqrt 2}$ and with $x:=a_1-\sqrt{1-a_1^2}$, ${\sqrt 3} x - 9 {\sqrt 2} x^3 >0$ if $x < \sqrt{{\sqrt 3}/(9 {\sqrt 2})} \simeq 0.37$. But $ x \le \frac{{\sqrt 2}-1}{{\sqrt 3}} \le 0.14$ satisfies this. Therefore $\delta$ is maximal for $a_1 = \frac 1 {{\sqrt 2}}$,
$\delta(\frac 1 {{\sqrt 2}}) = \frac{{\sqrt 2}}{3 {\sqrt 3}} - \frac 1 4 \le 0.0222$. Thus $M \le \frac{0.0222}{0.3098} < 0.072 < 0.0755$, proving \eqref{eq4.10}. We conclude that Theorem \ref{th1} is also valid for $n=4$. \\
Note that the maximum of $A_1( \cdot , t)$ is not attained in the second case $\frac 1 {{\sqrt 3}} \le t < a_2 < a_1$. Thus the only maximal hyperplanes in $B_1^n$ are the $(n-1)$-dimensional coordinate planes.
\end{proof}

{\it Remark.}  For $n \ge 5$, $f(a_1,t)$ is actually increasing for $a_1 \in [\frac 1 {{\sqrt 2}}, \sqrt{\frac 2 3}]$ and for all $t \in [\frac 1 {{\sqrt 3}},\frac 1 {{\sqrt 2}}]$ has its maximum at $a_1 = \sqrt{1-t^2}$ with $f(a_1,t) \le \frac{(\sqrt{1-t^2}-t)^{n-1}}{(1-2 t^2) \sqrt{1-t^2}^{n-2}}$. Thus for $n \ge 5$, $a^{(2)}$ is not a local extremum of $A(\cdot , t)$, $\frac 1 {\sqrt 3} \le t \le \frac 1 {\sqrt 2}$ and for $\tilde{a} = (\sqrt{1-t^2},t,0, \cdots ,0)$, $A_1(a^{(2)},t) < A_1(\tilde{a},t) < A_1(a^{(1)},t)$. For $n=4$, $f(a_1,t)$ is first decreasing in $a_1$, then increasing. For $t < 0.629$, its maximum occurs at $a = \frac 1 {{\sqrt 2}}$, for $t > 0.629$ at $\sqrt{1-t^2}$. For $n=3$, $f(a_1,t)$ is decreasing in $a_1$ for all $t \in [\frac 1 {{\sqrt 3}}, \frac 1 {{\sqrt 2}}]$. We do not prove these statements, since we do not need them.

\vspace{0,5cm}

\section{Conjectures}

We already mentioned that $a^{(n)}$ does not yield the maximal hyperplane at distance $t = \frac 1 {{\sqrt n}}$, if $n \ge 7$, different from the cases $n=2 ,3$. For $n \ge 7$, $A_1(a^{(1)},\frac 1 {{\sqrt n}}) > A_1(a^{(n)},\frac 1 {{\sqrt n}})$. For $n \ge 7$, we conjecture that there are $\gamma_n < \beta_n < \alpha_n$ such that $A_1( \cdot ,t)$ is maximal for
$$\left\{\begin{array}{c@{\quad}l}
a^{(n)}  \; \ , & \; \gamma_n \le t \le \beta_n \\
a^{(2)}  \; \ , & \; \beta_n \le t \le \alpha_n \\
a^{(1)}  \; \ , & \; \alpha_n \le t < 1
\end{array}\right\}. $$
The numbers $\alpha_n$ and $\beta_n$ are determined as the larger of the two solutions for $t$ of $A_1(a^{(1)},t) = A_1(a^{(2)},t)$ and $A_1(a^{(2)},t) = A_1(a^{(n)},t)$, respectively. These are of order $O(\frac 1 n)$, since e.g. $A_1(a^{(1)},\frac c {n+2}) = A_1(a^{(2)},\frac c {n+2}) + O(\frac 1 n)$ yields the equation $\exp(({\sqrt 2}-1)c) = \frac 1 {{\sqrt 2}} +c$, which has as its larger solution $c \simeq 3.4256$ (and a smaller solution $c \simeq 0.5474$). We have that $\alpha_n \le \frac 1 {{\sqrt n}}$ for $n \ge 7$. Thus for these $n$, $a^{(1)}$ is conjectured to yield the absolute maximum for at least all $\frac 1 {{\sqrt n}} \le t < 1$. In K\"onig \cite{Ko} it was shown that $a^{(1)}$ yields at least a local maximum of $A_1(\cdot,t)$ for all $\frac 3 {n+2} < t < 1$ and a local minimum for $0 < t < \frac 3 {n+2}$. Thus we conjecture that for only slightly larger values of $t$ than $\frac 3 {n+2}$, this local maximum is a global maximum, if $n \ge 7$. We also think that for $n \ge 4$, there is no continuous curve of maximal vectors between $a^{(n)}$ and $a^{(2)}$, different from dimension $n=3$. The numbers $\beta_n$ are only slightly smaller than the values of $\alpha_n$, like $\frac c {n+2}$ with $c = 3 + \varepsilon$. For $t$ close to zero and $n = 2, 3$, the vector $a^{[t]} := (\sqrt{1-(n-1) t^2},t, \cdots,t) \in S^{n-1}$ yields the absolute maximum of $A_1(\cdot,t)$. We conjecture that there are small numbers $\delta_n > 0$ such that $a^{[t]}$ yields the maximal hyperplane for $0 < t < \delta_n$. For $\delta_n < t < \gamma_n$ there should be a continuous curve of maximal vectors joining $a^{[t]}$ to $a^{(n)}$. For $n=6$, basically the same should hold, but there will be a discontinuity in $A_1(a_{max}, \cdot)$ at $t= \frac 1 {{\sqrt 6}}$. For $n=4, 5$, in between $a^{(n)}$ and $a^{(2)}$ there may be $t$-intervals where $a^{(3)}$ ($n=4$) or $a^{(4)}$ and $a^{(3)}$ ($n=5$) yield the maximal hyperplane. \\

Since $A_1(a^{(n)},t)=0$ for $t > \frac 1 {{\sqrt n}}$, the absolute minimum is only of interest for $0 < t <\frac 1 {{\sqrt n}}$. We conjecture that there are numbers $0 < c_n < d_n \le \frac 1 {{\sqrt n}}$ such that the absolute minimum of $A_1(\cdot,t)$ is determined by
$$\left\{\begin{array}{c@{\quad}l}
a^{(n)}  \; \ , & \; \ 0 \le t \le c_n \\
a^{(1)}  \; \ , & \; c_n \le t < d_n
\end{array}\right\} \ , $$
which is true for $n=2, 3$. The numbers $c_n$ are the smaller of the two $t$-solutions of $A_1(a^{(n)},t) = A_1(a^{(1)},t)$. They should be of order $\frac b {n+2}$ with $b \in (0.6,0.7)$. For $n=4, 5$ probably $d_n = \frac 1 {{\sqrt n}}$, as it is for $n=2,3$. For $n \ge 8$, $d_n$ might be the solution of $A_1(a^{(1)},t) = A_1(a^{[t]},t)$ with $d_n = \frac{5/2} n +O(\frac 1 {n^2})$, and then $a^{[t]}$ might yield the minimum for $d_n \le t \le \frac 1 {{\sqrt n}}$. For $n=6, 7$, probably $a^{(n-1)}$ yields the minimum for $d_n \le t \le \frac 1 {{\sqrt n}}$. \\
Concerning the vector $a^{[t]}$, we mention that for the $t$-derivatives
\begin{align*}
A_1'(a^{[t]},0) &= A_1'(a^{(1)},0) = -(n-1) \frac {2^{n-1}}{(n-1)!} \ , \\
A_1''(a^{[t]},0) &= (n-1)(n+1) \frac {2^{n-1}}{(n-1)!} > A_1''(a^{(1)},0) = (n-1)(n-2) \frac {2^{n-1}}{(n-1)!}
\end{align*}
and
\begin{align*}
A_1(a^{[t]},\frac 1 {{\sqrt n}}) &= A_1(a^{(n)},\frac 1 {{\sqrt n}}) = \frac{{\sqrt n}}{2^{n-1}} < A_1(a^{(1)},\frac 1 {{\sqrt n}}) = (1- \frac 1 {{\sqrt n}})^{n-1} \ , n \ge 7 \ , \\
A_1'(a^{[t]},\frac 1 {{\sqrt n}}) &= A_1'(a^{(n)},\frac 1 {{\sqrt n}}) = -\frac{n(n-1)(n-2)}{2^n} \ .
\end{align*}
Therefore, for $n \ge 7$, $d_n < \frac 1 {\sqrt n}$ and there has to be a change from $a^{(1)}$ to a different minimum vector, and most likely it will not be $a^{(n)}$ since probably $A_1''(a^{[t]},\frac 1 {{\sqrt n}}) < A_1''(a^{(n)},\frac 1 {{\sqrt n}})$ : The latter is true at least for $n=8$. \\

\vspace{0,5cm}

\section{Sections of the polydisc}

In dimension $2$, the real $l_1$-ball $B_1^2$ and the real $l_\infty$-ball $B_\infty^2$ are isometric. For $n \ge 2$, $a \in S^{n-1}$ and $t \ge 0$, let
$$A_\infty(a,t) = \vol_{n-1} \{ \ x \in B_\infty^n \ | \ <x,a> =t \ \} \ . $$
Put $a^{[t]} := ( (\sqrt{2-t^2}+t)/2 , (\sqrt{2-t^2}-t)/2 ) \in S^1$. Then for $n=2$ we have the following extrema of $A_\infty(\cdot,t)$ : \\
$$ Maximum :  \left\{\begin{array}{c@{\;}l}
a^{[t]}  , &  \; \quad 0 \le t \le 1 \\
a^{(2)}  , &  \frac 3 4 {\sqrt 2} \le t < {\sqrt 2}
\end{array}\right\} \ , \
Minimum : \left\{\begin{array}{c@{\;}l}
a^{(1)}  , & \ \ 0 \le t \le {\sqrt 2} -1 \\
a^{(2)}  , & \ {\sqrt 2} -1 \le t < 1
\end{array}\right\} . $$
For $1 < t < \frac 3 4 {\sqrt 2}$, the maximal vector rotates from $a^{(1)}=a^{[1]}$ to $a^{(2)}$, with a discontinuity of $A_\infty(a_{max},\cdot)$  at $t=1$, cf. K\"onig, Koldobsky \cite{KK1}. If $a = (a_1,a_2) \in S^1$, this is a consequence of
\begin{equation}\label{eq6.1}
A_\infty(a,t) = \left\{\begin{array}{c@{\quad}l}
\quad \frac 2 {a_1}  \quad \quad , & \quad \quad \quad 0 \le t \le a_1 - a_2 \\
\frac{a_1+a_2-t}{a_1 a_2} \; \ , & \; a_1-a_2 \le t < a_1+a_2
\end{array}\right\}.
\end{equation}
We will now study the complex analogue of this. Let $\D := \{ z \in \C \ | \ |z| \le 1 \}$ be the unit disc in $\C$ and $B_\infty^n := \D^n$ be the $n$-dimensional polydisc (complex cube). Let $a \in \C^n$, $||a||_2=1$ and $t \in \C$. Define similarly as above
$$A_{\infty,\C}(a,t) := \vol_{2(n-1)} \{ z \in B_\infty^n \ | \ <a,z> =t \} \ ,$$
where $<\cdot,\cdot>$ is the scalar product in $\C^n$ and $\vol_{2(n-1)}$ is the real $2(n-1)$-dimensional Lebesgue measure, identifying $\C^n$ with $\R^{2n}$. Since the hyperplane section is invariant under coordinate-wise multiplications by $\exp(i \theta_j)$, we may assume that $t \ge 0$ and $a \in S^{n-1}$ with $1 \ge a_1 \ge a_2 \ge \cdots \ge a_n \ge 0$. The there is the following formula for the complex parallel section function of the polydisc:
\begin{equation}\label{eq6.2}
A_{\infty,\C}(a,t) = \pi^{n-1} \frac 1 2 \int_0^\infty \prod_{k=1}^n  j_1(a_k s) J_0(t s) \ s \ ds \ ,
\end{equation}
where $j_1(x) := 2 \frac{J_1(x)} x$ and $J_0$ and $J_1$ are the usual Bessel functions, cf. Oleszkiewicz, Pelczy\'nski \cite{OP} or K\"onig, Koldobsky \cite{KK2}. For $t=0$, Oleszkiewicz and Pelczy\'nski proved the complex analogue of Ball's theorem by showing that $a^{(2)}$ yields the maximal (central) complex hyperplane in the $n$-polydisc. We will now evaluate \eqref{eq6.2} for $n=2$ in the non-central case $t>0$ and prove the following complex analogue of \eqref{eq6.1}:

\begin{proposition}\label{prop6.1}
Let $1 \ge a_1 \ge a_2 \ge 0$, $a_1^2+a_2^2 = 1$, $0 \le t \le {\sqrt 2}$. Then for $n=2$:
\begin{equation}\label{eq6.3}
A_{\infty,\C}(a,t) =  \left\{\begin{array}{c@{\quad}l}
\quad \frac \pi {a_1^2}  \quad \quad \quad , & \quad \quad \quad 0 \le t \le a_1 - a_2 \\
\phi(a_1,a_2,t) \; \ , & \; a_1-a_2 \le t < a_1+a_2
\end{array}\right\},
\end{equation}
\begin{align*}
\phi(a_1,a_2,t) & =\frac{\pi}{2 a_1^2} + \frac 1 {2a_1^2a_2^2} [ \arcsin(\frac{1-t^2}{2a_1a_2}) \\
 & +(a_1^2-a_2^2) \arcsin(\frac{4a_1^2a_2^2-(1-t^2)}{2a_1a_2t^2})-\sqrt{4a_1^2a_2^2-(1-t^2)^2} ] \ .
\end{align*}
\end{proposition}

{\it Remark.} Note that $a_1-a_2 \le t \le a_1+a_2$ implies that the square root in $\phi(a_1,a_2,t)$ is well-defined and that the arguments in the $\arcsin$-terms are between $-1$ and $1$. It also implies $a_1 \le \frac 1 2 (\sqrt{2-t^2} + t)$.

\begin{proof}
The addition formula for the normalized Bessel function $j_1$ states
$$j_1(a_1s) j_1(a_2s) = \frac 2 {\pi} \int_0^\pi j_1(\sqrt{a_1^2+a_2^2-2a_1a_2 \cos(\theta)} s) \sin(\theta)^2 d \theta \ , $$
cf. Watson \cite{W}, 11.41. The function $f:[0,\pi] \to [a_1-a_2,a_1+a_2]$, \\
$f(\theta) := \sqrt{a_1^2+a_2^2-2a_1a_2 \cos(\theta)}$ is bijective and we substitute $u=f(\theta)$, \\
$u^2 = (a_1-a_2)^2+4a_1a_2 \sin(\frac{\theta} 2)^2$,
$\sin(\frac{\theta} 2)^2 = \frac{u^2-(a_1-a_2)^2}{4a_1a_2}$, $\cos(\frac{\theta} 2)^2 =\frac{(a_1+a_2)^2-u^2}{4a_1a_2}$, \\
$\sin(\theta)^2 = 4 \sin(\frac{\theta} 2)^2 \cos(\frac{\theta} 2)^2 = \frac 1 {(2a_1a_2)^2} ((a_1+a_2)^2-u^2)(u^2-(a_1-a_2)^2)$. Further $u du = a_1a_2 \sin(\theta) d \theta$. Hence
\begin{align*}
j_1(a_1s) j_1(a_2s) &= \frac 2 {\pi} \int_{a_1-a_2}^{a_1+a_2} j_1(us) \sin(\theta) \frac{u du}{a_1a_2} \\
& = \frac 1 {\pi a_1^2a_2^2} \int_{a_1-a_2}^{a_1+a_2} \sqrt{((a_1+a_2)^2-u^2)(u^2-(a_1-a_2)^2)} j_1(us) u du \ .
\end{align*}
Therefore by \eqref{eq6.2} for $n=2$
\begin{align*}
& A_{\infty,\C}(a,t) = \\
& = \frac 1 {2 a_1^2a_2^2} \int_{a_1-a_2}^{a_1+a_2} \sqrt{((a_1+a_2)^2-u^2)(u^2-(a_1-a_2)^2)} \Big(\int_0^\infty j_1(us)J_0(ts) s ds \Big) \ u du \ .
\end{align*}
However,
$$\int_0^\infty j_1(us) J_0(ts) s ds = \frac 2 u \int_0^\infty J_1(us) J_0(ts ) ds =  \left\{\begin{array}{c@{\quad}l}
\frac 2 {u^2}  \; , & \quad   u > t\\
\; 0 \; \  , & \quad u < t
\end{array}\right\} \ , $$
cf. Watson \cite{W}, 13.42. Hence
$$A_{\infty,\C}(a,t) = \frac 1 {a_1^2a_2^2} \int_{\max(t,a_1-a_2)}^{a_1+a_2} \sqrt{((a_1+a_2)^2-u^2)(u^2-(a_1-a_2)^2)} \ \frac {du} u \ . $$
Now for $0 \le B < u < C$,
\begin{align*}
& \int \sqrt{(C^2-u^2)(u^2-B^2)} \frac {du} u = \frac 1 2 \sqrt{(C^2-u^2)(u^2-B^2)} \\
& \quad +\frac{C^2+B^2} 4 \arcsin(\frac{2 u^2 -(C^2+B^2)}{C^2-B^2}) +\frac{CB} 2 \arcsin(\frac{2C^2B^2 - (C^2+B^2) u^2}{u^2 (C^2-B^2)}) \ .
\end{align*}
Thus for $t \ge a_1-a_2$ with $a_1^2+a_2^2 = 1$,
\begin{align*}
A_{\infty,\C}(a,t) & = \frac 1 {a_1^2a_2^2} \Big( -\frac 1 2 \sqrt{((a_1+a_2)^2-t^2)(t^2-(a_1-a_2)^2)} + \frac 1 2 [\frac{\pi} 2 - \arcsin(\frac{t^2-1}{2a_1a_2})] \\
& + \frac 1 2 (a_1^2-a_2^2) [- \frac {\pi} 2 - \arcsin(\frac{(a_1^2-a_2^2)^2-t^2}{2a_1a_2t^2})] \Big)  \\
& = \frac{\pi}{2 a_1^2} + \frac 1 {2a_1^2a_2^2} \Big( \arcsin(\frac{1-t^2}{2a_1a_2}) \\
& + (a_1^2-a_2^2) \arcsin(\frac{4a_1^2a_2^2-(1-t^2)}{2a_1a_2t^2})- \sqrt{4a_1^2a_2^2-(1-t^2)^2} \Big) \ .
\end{align*}
For $0 \le t \le a_1-a_2$, we have to integrate from $a_1-a_2$ to $a_1+a_2$, i.e. put $t = a_1-a_2$ in the last formula. Then
$A_{\infty,\C}(a,t) = \frac {\pi}{a_1^2}$, since in this case the term in big brackets is $\pi a_2^2$. This proves \eqref{eq6.3}.
\end{proof}

\begin{corollary}\label{cor6.2}
For $n=2$
\begin{align*}
A_{\infty,\C}(a^{(1)},t) &= \pi  \quad , \quad 0 \le t \le 1 \ , \\
A_{\infty,\C}(a^{(2)},t) &= \pi + 2 ( \arcsin(1-t^2) - t \sqrt{2-t^2} ) \ , \ 0 \le t \le {\sqrt 2} \ \\
A_{\infty,\C}(a^{[t]},t) &= \frac{2 \pi}{1 + t \sqrt{2-t^2}} \quad , \quad 0 \le t \le 1 \ ,
\end{align*}
with $a^{[t]} = (\frac 1 2 (\sqrt{2-t^2}+t),\frac 1 2 (\sqrt{2-t^2}-t))$ as in the real case above.
\end{corollary}

Note that for $a^{[t]}$ we have $a_1-a_2 =t$, so the simpler formula in Proposition \ref{prop6.1} applies. \\

For $a^{(2)}$ and $t>0$, we are in the case $0 = a_1-a_2 < t$, thus $A_{\infty,\C} = \phi$.
\begin{align*}
\frac{\partial \phi (a_1,\sqrt{1-a_1^2},t)}{\partial a_1} &= \frac 1 {a_1^3a_2^4} \Big( (a_1^2-a_2^2) [\arcsin(\frac{1-t^2}{2a_1a_2})-\sqrt{4a_1^2a_2^2-(1-t^2)^2}   ]  \\
& + (a_1^4 + a_2^4) \arcsin(\frac{4a_1^2a_2^2-(1-t^2)}{2a_1a_2 t^2}) \Big) - \frac {\pi} {a_1^3} \ .
\end{align*}
In particular, $\frac{\partial \phi}{\partial a_1} (\frac 1 {{\sqrt 2}},\frac 1 {{\sqrt 2}},t) = 0$ : $a^{(2)}$ is a critical point of $A_{\infty,\C}(\cdot,t)$ for all $0 \le t \le {\sqrt 2}$. The second derivative of $A_{\infty,\C} = \phi$ at $a_1 = \frac 1 {{\sqrt 2}}$ satisfies
$$\lim_{a_1 \searrow 1/{\sqrt 2}} \frac{\partial^2 \phi}{\partial a_1^2} (a_1,\sqrt{1-a_1^2},t) = 32 [ \frac{\pi} 2 + \arcsin(1-t^2) - \sqrt{2-t^2} (t + \frac 1 {2t}) ] =: 32 \psi(t) \ . $$
The function $\psi$ has two zeros in $[0,{\sqrt 2}]$, $t_1 \simeq 0.30074$, $t_2 \simeq 1.07425$: The derivative is $\psi'(t) = \frac{1-4 t^2+ 2 t^4}{t^2 \sqrt{2-t^2}}$, i.e. $\psi$ is increasing in $[0,\frac 1 2 \sqrt{4 - 2 {\sqrt 2}}] \simeq [0,0.5412]$ and in $[\frac 1 2 \sqrt{4 + 2 {\sqrt 2}},{\sqrt 2}] \simeq [1.3066,{\sqrt 2}]$ and decreasing in between. With $\lim_{t \searrow 0} \psi(t) = - \infty$, $\psi(\frac 1 2 \sqrt{4 - 2 {\sqrt 2}}) = \frac {3 \pi} 4 - \sqrt 2 - \frac 1 2 \simeq 0.442 >0$, $\psi(\frac 1 2 \sqrt{4 + 2 {\sqrt 2}}) = \frac {\pi} 4 -{\sqrt 2} + \frac 1 2 \simeq -0.129 < 0$, $\psi$ has exactly two zeros in $[0,{\sqrt 2}]$. Therefore \newline $\lim_{a_1 \searrow 1/{\sqrt 2}} \frac{\partial^2 \phi}{\partial a_1^2} (a_1,\sqrt{1-a_1^2},t) < 0$ if and only if $0 <t < t_1$ or $t_2 < t <{\sqrt 2}$. Hence for $t \le t_1 \simeq 0.30074$ and for $1.07425 \simeq t_2 \le t < {\sqrt 2}$, $A_{\infty,\C}(\cdot,t)$ has (at least) a local maximum in $a^{(2)}$. \\

The derivative of $\phi$ with respect to $t$ is $\frac{\partial \phi}{\partial t} (a_1,a_2,t) = - \frac{\sqrt{4a_1^2a_2^2 - (1-t^2)^2}}{a_1^2a_2^2 t} < 0$, and
$\frac{\partial \phi}{\partial t} (\frac 1 {{\sqrt 2}},\frac 1 {{\sqrt 2}},t) = - 4 \sqrt{2-t^2} < 0$. Therefore
$\phi(\frac 1 {{\sqrt 2}},\frac 1 {{\sqrt 2}},t) = A_{\infty,\C}(a^{(2)},t) = A_{\infty,\C}(a^{(1)},t) = \pi$ has exactly one solution $\bar{t}$,  $\bar{t} \simeq 0.57130$. For $t > \bar{t}$, $A_{\infty,\C}(a^{(2)},t) <  A_{\infty,\C}(a^{(1)},t)$, so $a^{(1)}$ cannot be the absolute minimum of  $A_{\infty,\C}(\cdot,t)$ for $t > \bar{t}$. \\

Plots of $\phi$ for various values of $t$ indicate that $a^{(2)}$ is, in fact, the absolute maximum of $A_{\infty,\C}(\cdot,t)$ for all $t \le t_1$ and for all $t \ge t_2$. For $t_1 < t < t_2$, the maximum is attained at some $(a_1,a_2)$ with $\frac 1 {{\sqrt 2}} < a_1 < 1$, $a_1$ increasing with $t$ in $(t_1,1)$ and decreasing with $t$ in $(1,t_2)$. This value of $a_1$ does not coincide with $\frac 1 2 (\sqrt{2-t^2}+t)$, so, different from the real case, $a^{[t]}$ will not yield the absolute maximum. In fact, for $0 < t < 0.389$, $A_{\infty,\C}(a^{[t]},t) < A_{\infty,\C}(a^{(2)},t)$. We note that $a^{[0]} = a^{(2)}$, but
$ \frac d {dt} A_{\infty,\C}(a^{[t]},t) (0) = - 2 \pi {\sqrt 2} < -4 {\sqrt 2} = \frac d {dt} A_{\infty,\C}(a^{(2)},t) (0)$.
The plots also indicate that $a^{(1)}$ yields the minimum ($\pi$) of $A_{\infty,\C}(\cdot,t)$ if $0 \le t \le \bar{t}$ and $a^{(2)}$ the minimum if $\bar{t} < t < 1$. This would mean:
$$ Maximum :  \left\{\begin{array}{c@{\quad}l}
a^{(2)} \ , &  \; \ 0 \le t \le t_1 \\
a^{(2)} \ , &  \ t_2 \le t < {\sqrt 2}
\end{array}\right\}  \; , \;
Minimum : \left\{\begin{array}{c@{\quad}l}
a^{(1)}  \; \ , & \;  0 \le t \le \bar{t} \\
a^{(2)}  \; \ , & \; \bar{t} \le t < 1
\end{array}\right\}  \ . $$
Between $t_1$ and $1$, the maximal vector seems to turn from $a^{(2)}$ to $a^{(1)}$ and between $1$ and $t_2$ back from $a^{(1)}$ to $a^{(2)}$.
Comparing this with the real situation, the case of the minimum seems similar, with ${\sqrt 2}-1$ being replaced by $\bar{t}$. However, in the maximum case, there is the range $0 \le t \le t_1$ where $a^{(2)}$ and not $a^{[t]}$ will yield the maximum. For $t \ge t_2$, the situation is similar, with $\frac 3 4 {\sqrt 2}$ replaced by $t_2$. \\

Trying to use \eqref{eq6.2} in a similar manner for $n>2$, say for $n=3$ with $a_1^2+a_2^2+a_3^2 =1$, $a_1 \ge a_2 \ge a_3$ leads to iterated integrals of the form
$$\frac 1 {{\pi}^2 (a_1a_2a_3)^2} \int_{a_1-a_2}^{a_1+a_2} g(a_1+a_2,u,a_1-a_2) \frac 1 u ( \int_{\max(t,|u-a_3|)}^{u+a_3} g(u+a_3,v,u-a_3) \frac 1 v dv ) du \ , $$
where $g(C,u,B) = \sqrt{(C^2-u^2)(u^2-B^2)}$, which seem difficult to evaluate.
\begin{corollary}
We have $\lim_{n \to \infty} \frac 1 {{\pi}^{n-1}} A_{\infty,\C}(a^{(n)},t) = 2 \exp(-2 t^2) \ .$ This is $  > \frac 1 {\pi} A_{\infty,\C}(a^{(2)},t)$ $(n=2)$ for all $0 < t < 0.6299$. For any $0 < t < 0.6299$, $a^{(2)}$ will not yield the maximal hyperplane in the normalized polydisc $\frac 1 {\pi^n} \D^n$ for $n > \frac 1 t$, $n \ge 5$: for these $t$ and $n$, $A_{\infty,\C}(a^{(n)},t) > A_{\infty,\C}(a^{(2)},t)$ holds.
\end{corollary}
For $t=0$, by Oleszkiewicz, Pelczy\'nski \cite{OP}, $a^{(2)}$ yields the maximal hyperplane for all $n \in \N$.
\begin{proof}
Since $j_1(x) = 1- \frac{x^2} 8 + O(x^4)$ for $x$ near $0$, cf. Watson \cite{W}, we have by \eqref{eq6.2}
$$\lim_{n \to \infty} \frac 1 {{\pi}^{n-1}} A_{\infty,\C}(a^{(n)},t) = \frac 1 2 \int_0^\infty \exp(- \frac{s^2} 8) J_0(t s) s ds = 2 \exp(-2 t^2) \ , $$
for the last equality cf. Watson \cite {W}, 13.3. This is larger than the value from Corollary \ref{cor6.2} ($n=2$)
$$\frac 1 {\pi} A_{\infty,\C}(a^{(2)},t) = 1 + \frac 2 {\pi} (\arcsin(1-t^2)-t \sqrt{2-t^2}) $$
for $t \le 0.63$ and for $t \ge 1.26$ and smaller in between. Dividing by the powers of $\pi$ normalizes the volume of the polydisc.  The difference between
$2 \exp(-2 t^2)$ and $1 + \frac 2 {\pi} (\arcsin(1-t^2)-t \sqrt{2-t^2})$ is bigger than $t$ for all $0 < t < \frac 1 5$. On the other hand, the estimate
$|j_1(s)| \le \exp(-\frac {s^2} 8 - \frac{s^4} {384})$ for all $0 \le s \le 4$, shown by Oleszkiewicz and Pelczy\'nski \cite{OP}, together with
$ |j_1(s)| \le \frac {5/3}{s^{3/2}}$ for $s \ge 4$ can be used to estimate the difference
$$ | \lim_{m \to \infty} \frac 1 {{\pi}^{m-1}} A_{\infty,\C}(a^{(m)},t) - \frac 1 {{\pi}^{n-1}} A_{\infty,\C}(a^{(n)},t) | $$
from above by $\frac 1 n$ for all $n \ge 3$. Therefore $\frac 1 {{\pi}^{n-1}} A_{\infty,\C}(a^{(n)},t) > \frac 1 {\pi} A_{\infty,\C}(a^{(2)},t)$ for all $n > \frac 1 t$, $0 < t < \frac 1 5$. For these $n$, the vector $a^{(2)}$ no longer yields the maximal hyperplane in the normalized polydisc. \\
\end{proof}
The real situation is slightly different: we have for $0 < t < 0.0347$
$$\frac 1 2 A_{\infty}(a^{(2)},t) = {\sqrt 2} - t > \sqrt{\frac 6 \pi} \exp(-\frac 3 2 t^2) =\lim_{n \to \infty} \frac 1 {2^{n-1}} A_{\infty}(a^{(n)},t) \ . $$
However, $A_{\infty}(a^{(2)},t) < A_{\infty}(a^{[t]},t)$ for all $0 < t <1$.
\vspace{0,5cm}

\section{Sections of the $l_1$-ball $B_1^2$}

In this section $B_1^2 := \{ z \in \C^2 \ | \ ||z|| \le 1 \}$ will denote the complex $l_1$-ball and for $a \in \C^2$, $||a|| =1$ and $t \in \C$
$$A_{1,\C} (a,t) := \vol_2 ( \{ z \in B_1^2 \ | \ < a , z > = t \} ) $$
will denote the volume of the complex hyperplane section having real dimension $2$. Here $< \cdot,\cdot>$ denotes the scalar product in $\C^2$. Since the section is invariant under coordinate-wise multiplications by $\exp(i \theta_j)$, we may assume that $t \ge 0$ and $1 \ge a_1 \ge a_2 \ge 0$, $a_1^2+a_2^2=1$. Then

\begin{proposition}\label{prop7.1}
Let $a \in S^1$, $1 \ge a_1 \ge a_2 \ge 0$ and $t \ge 0$. Define
$$f(a_1,a_2,t;r) := \frac{t^2-a_2^2+2a_2^2r+(a_1^2-a_2^2)r^2}{2t a_1 r} \ , \ r \in [0,1] \ . $$
Then
$$ A_{1,\C}(a,t) = \left\{\begin{array}{c@{\quad}l}
& \pi (\frac{1-t/a_2}{a_1+a_2})^2 + \frac 2 {a_2^2} \int_{\frac{a_2-t}{a_1+a_2}}^{\frac{a_2+t}{a_1+a_2}} \arccos(f(a_1,a_2,t;r)) \ r dr \ , \ 0 \le t \le a_2 \ , \\
& \frac 2 {a_2^2} \int_{\frac{t-a_2}{a_1-a_2}}^{\frac{t+a_2}{a_1+a_2}} \arccos(f(a_1,a_2,t;r)) \ r dr \ , \ a_2 < t \le a_1 \ . \\
\end{array}\right\} $$
We have, in particular, with $a^{[t]} = (\sqrt{1-t^2},t)$
\begin{align*}
A_{1,\C}(a^{(1)},t) &= \pi (1-t)^2 \quad , \quad 0 \le t \le 1 \\
A_{1,\C}(a^{(2)},t) &= \frac {\pi} 2 \sqrt{1-2 t^2} \quad , \quad 0 \le t \le \frac 1 {{\sqrt 2}} \\
A_{1,\C}(a^{[t]},t) &=\frac{2(1+t^2)}{(1-2t^2)^2} \arccos(\frac t {\sqrt{1-t^2}}) - \frac {6t}{(1-2t^2)^{3/2}} \; , \;  0 \le t \le \frac 1 {{\sqrt 2}} \ .
\end{align*}
\end{proposition}

\begin{proof}
(i) Let $a \in S^1$, $1 \ge a_1 \ge a_2 \ge 0$ and $t \in \C$. Identifying $\C$ with $\R^2$ for the purpose of the Fourier transformation, we use $\cdot$ for the scalar product in $\R^2$. We get, similarly as in Section 2,
\begin{align*}
\hat{A}_{1,\C}(a,s) &= \frac 1 {2 \pi} \int_{\R^2} \int_{<a,z> = t} \chi_{B_1^2}(z) \exp(-i t \cdot s) \ dz \ dt  \\
& = \frac 1 {2 \pi} \int_{\R^4} \chi_{B_1^2}(z) \exp(-i s \ \cdot <a,z>) \ dz \ , \ s \in \R^2 \ .
\end{align*}
Let $\D_r := \{ z \in \C \ | \ |z| \le r \}$. Let $z=(z_j)$, $z_j = x_j + i y_j$, $j=1,2$. Then, identifying $\C$ with $\R^2$,
$s \ \cdot <a,z> = a_1 (s_1x_1-s_2y_1) +a_2 (s_1x_2-s_2y_2)$ and
\begin{align*}
\int_{\D_{1-|z_1|}} \exp(-i a_2 (s_1x_2-s_2y_2) ) dz_2 &= \int_0^{1-|z_1|} \Big(\int_0^{2 \pi} \cos( \sqrt{s_1^2+s_2^2} \ a_2 u \cos(\phi)) d \phi \Big)  \ u du \\
& = 2 \pi \int_0^{1-|z_1|} J_0( \sqrt{s_1^2+s_2^2} \ a_2 u ) \ u du \\
& = 2 \pi \frac{J_1( \sqrt{s_1^2+s_2^2} \ a_2 (1-|z_1|) )}{\sqrt{s_1^2+s_2^2} \ a_2} (1-|z_1|) \ ,
\end{align*}
where we used that $\int J_0(x) x dx = J_1(x) x$, cf. Abramowitz, Stegun \cite{AS}, 9.1.30 or Watson \cite{W}, 3.13. Hence
\begin{align*}
\hat{A}_{1,\C}(a,s) &= \frac 1 {2 \pi} \int_{\D} \exp(-i a_1(s_1x_1-s_2y_1) ) ( \int_{\D_{1-|z_1|}} \exp(-i a_2(s_1x_2-s_2y_2)) \ dz_2 ) \ dz_1 \\
& =\int_0^1 \Big(\int_0^{2 \pi} \cos(\sqrt{s_1^2+s_2^2} \ a_1 r cos(\psi) ) d\psi \Big) \frac{J_1( \sqrt{s_1^2+s_2^2} \ a_2 (1-r) )}{\sqrt{s_1^2+s_2^2} \ a_2} (1-r) r dr \\
& = 2 \pi \int_0^1 J_0(\sqrt{s_1^2+s_2^2} \ a_1 r) \frac{J_1( \sqrt{s_1^2+s_2^2} \ a_2 (1-r) )}{\sqrt{s_1^2+s_2^2} \ a_2} (1-r) r dr
\end{align*}
The inverse Fourier transform yields that
\begin{align*}
A_{1,\C}(a,t) &= \frac 1 {2 \pi} \int_{\R^2} \exp(i t \cdot s) \hat{A}_{1,\C}(a,s) ds \\
& = \int_0^1 \Big( \int_{\R^2} \exp(i t\cdot s) J_0(\sqrt{s_1^2+s_2^2} \ a_1 r) \frac{J_1( \sqrt{s_1^2+s_2^2} \ a_2 (1-r) )}{\sqrt{s_1^2+s_2^2} \ a_2} d(s_1,s_2) \Big) (1-r) r dr \\
& = \int_0^1 \Big( \int_0^\infty (\int_0^{2 \pi} cos(|t| v cos(\gamma)) d\gamma) J_0(v a_1 r) \frac{J_1(v a_2(1-r))}{v a_2} v dv \Big) (1-r) r dr \\
& = 2 \pi \int_0^1 \Big( \int_0^\infty J_0(|t| v) J_0(a_1 r v) \frac{J_1(a_2(1-r)v)}{a_2} dv \Big) (1-r) r dr \ ,
\end{align*}
where we used polar coordinates in $\R^2$, $s_1=v \cos(\gamma)$, $s_2=v \sin(\gamma)$, $v = \sqrt{s_1^2+s_2^2}$. Applying the addition formula for the Bessel function $J_0$,
$$J_0(c v)J_0(d v) = \frac 1 {\pi} \int_0^{\pi} J_0(\sqrt{c^2+d^2-2 c d \cos(\lambda)} v) d\lambda \ , $$
cf. Watson \cite{W}, 11.41 and the formula
$$\int_0^\infty J_0(Cv) J_1(Dv) dv = \left\{\begin{array}{c@{\quad}l}
\frac 1 D \ , &  \ C < D \\
\ 0 \ , &  \  C > D
\end{array}\right\}  \ , $$
cf. Watson \cite{W}, 13.42, the inner integral can be evaluated, and with $t = |t|$
$$\int_0^\infty J_0(t v) J_0(a_1 r v) \frac{J_1(a_2(1-r)v)}{a_2} dv = \frac B {\pi a_2^2 (1-r)} \ , $$
where
$$B = \left\{ \begin{array}{c@{\quad}l}
\quad 0  \quad  &  , \quad a_2(1-r) < |t-a_1 r| \\
\quad \pi  \quad  & , \quad  t+a_1 r < a_2 (1-r) \\
\arccos(f(a_1,a_2,t;r)) \  & , \quad |t-a_1 r| < a_2 (1-r) < t +a_1 r
\end{array} \right\} \ , $$
cf. also Watson \cite{W}, 13.46. Here $f$ is the function in the statement of Proposition \ref{prop7.1}. We note that $-1 \le f(a_1,a_2,t;r) \le 1$ if and only if $|t-a_1 r| < a_2 (1-r) < t +a_1 r$. \\

(ii) Suppose now that $ 0 < t \le a_2$. Then $f$ is increasing in $r$ since
\begin{equation}\label{eq7.1}
f(a_1,a_2,t;r) = \frac{t^2-a_2^2}{2t a_1 r} + \frac {a_2^2}{t a_1} + \frac{a_1^2-a_2^2}{2 t a_1} r \ ,
\end{equation}
with $f(a_1,a_2,t;\frac{a_2-t}{a_1+a_2}) = -1 < f(a_1,a_2,t;\frac{a_2+t}{a_1+a_2}) = 1$. Thus for $r \in [\frac{a_2-t}{a_1+a_2},\frac{a_2+t}{a_1+a_2}]$, the third term for $B$ applies. For $r \in [0,\frac{a_2-t}{a_1+a_2}]$, $t+a_1 r < a_2 (1-r)$, so $B = \pi$ then. We find
\begin{align*}
A_{1,\C}(a,t) &= 2 \pi \frac 1 {a_2^2} \int_0^{\frac{a_2-t}{a_1+a_2}} r dr + \frac 2 {a_2^2} \int_{\frac{a_2-t}{a_1+a_2}}^{\frac{a_2+t}{a_1+a_2}} \arccos(f(a_1,a_2,t;r)) \ r dr \\
& = \pi (\frac{1-t/a_2}{a_1+a_2})^2 + \frac 2 {a_2^2} \int_{\frac{a_2-t}{a_1+a_2}}^{\frac{a_2+t}{a_1+a_2}} \arccos(f(a_1,a_2,t;r)) \ r dr \ .
\end{align*}

(iii) Now consider $a_2 < t \le a_1$. Let $r_0:= \sqrt{\frac{t^2-a_2^2}{a_1^2-a_2^2}}$, $r_0$ is $<1$. By \eqref{eq7.1}, $f$ as a function of $r$ is decreasing in $(0,r_0)$ and increasing in $(r_0,1)$. We have $f(a_1,a_2,t;r) \ge 0$ for all $r \in [0,1]$ and $f(a_1,a_2,t;\frac{t-a_2}{a_1-a_2}) = 1$ with
$\frac{t-a_2}{a_1-a_2} < r_0$ and $f(a_1,a_2,t;\frac{t+a_2}{a_1+a_2}) = 1$ with $r_0 < \frac{t+a_2}{a_1+a_2}$. Thus the integration of the $\arccos$-term runs between $\frac{t-a_2}{a_1-a_2}$ and $\frac{t+a_2}{a_1+a_2}$. The second term is impossible here, since $a_2 (1-r) \le a_2 < t \le t +a_1 r$. Therefore
$$A_{1,\C}(a,t) = \frac 2 {a_2^2} \int_{\frac{t-a_2}{a_1-a_2}}^{\frac{t+a_2}{a_1+a_2}} \arccos(f(a_1,a_2,t;r)) \ r dr \ , \ a_2 < t \le a_1 \ , $$
which proves the basic formulas of Proposition \ref{prop7.1}. \\

(iv) To determine $A_{1,\C}(a^{(1)},t)$ for $0 < t < 1$, choose $0 < a_2 < t < a_1 < 1$ with $a_1^2+a_2^2 =1$. We will take the limit of $A_{1,\C}((a_1,a_2),t)$ as $a_1 \to 1$ and $a_2 \to 0$. We are in case 2, $a_2 < t < a_1$. The length of the interval of integration is
$\frac{t+a_2}{a_1+a_2} - \frac{t-a_2}{a_1-a_2} = 2 a_2 \frac{a_1-t}{a_1^2-a_2^2} \to 0$.  The function $f$ has the $r$-derivative values
$$\frac{\partial f}{\partial r}(a_1,a_2,t;r)_\Big{|r =\frac{t+a_2}{a_1+a_2}} = \frac{a_2}{t a_1} \frac{a_1-t}{t+a_2} (a_1+a_2) \ , $$
$$\frac{\partial f}{\partial r}(a_1,a_2,t;r)_\Big{|r =\frac{t-a_2}{a_1-a_2}} = - \frac{a_2}{t a_1} \frac{a_1-t}{t-a_2} (a_1-a_2) \ , $$
with $f$ having value $1$ in both points. In these points $\frac{\partial f}{\partial r}$ is approximately $\pm \frac{a_2}{a_1}\frac{a_1-t}{t^2}+O(a_2)$. The minimum of $f$ in $r=r_0$ is $1-\frac{(1-t)^2}{2t^2} a_2^2 + O(a_2^4)$. Define $g(x) := \frac{a_1^2-a_2^2}{2 t^2 a_1} x^2 - a_2 \frac{a_1-t}{t^2 a_1} x$. Then
$g(0) = g(2 a_2 \frac{a_1-t}{a_1^2-a_2^2}) = 0$, $g'(0) = - \frac{a_2}{a_1}\frac{a_1-t}{t^2}$, $g'(2 a_2 \frac{a_1-t}{a_1^2-a_2^2}) = \frac{a_2}{a_1}\frac{a_1-t}{t^2}$, as for $(f-1)$ on the endpoints up to $O(a_2)$, except that the interval is shifted to $0$. In the middle of the $g$-interval, $g$ takes value approximately $- \frac{(1-t)^2}{2 t^2} a_2^2$ up to $O(a_2^4)$, just as the minimum of $(f-1)$. Near $y=1$ we have $\arccos(y) = \sqrt{2 (1-y)} + O((\sqrt{2 (1-y)})^3)$. In the integral we have $r = \frac t {a_1} + O(a_2)$. With $\arccos(1+g(x)) \simeq \sqrt{- 2 g(x)}$ we get
\begin{align*}
A_{1,\C}((a_1,a_2),t) & = \frac 2 {a_2^2} (\frac t {a_1} + O(a_2)) \int_0^{2 a_2 \frac{a_1-t}{a_1^2-a_2^2}} \arccos(1+g(x)) (1+O(a_2)) \  dx \\
& = \frac {2t}{a_1 a_2^2} \frac{{\sqrt 2}}{\sqrt{a_1} t} \int_0^{2 a_2 \frac{a_1-t}{a_1^2-a_2^2}} {\sqrt x} \sqrt{a_2(a_1-t) - (a_1^2-a_2^2)/2 \ x} \ dx  \ (1+O(a_2))  \\
& = \frac{2 {\sqrt 2}}{a_1^{3/2} a_2^2} \Big(\frac{\pi}{2 {\sqrt 2}} \frac{(a_1-t)^2}{(a_1^2-a_2^2)^{3/2}} a_2^2 \Big) \ (1+O(a_2)) \\
& = \frac{\pi}{a_1^{3/2}} \frac{(a_1-t)^2}{(a_1^2-a_2^2)^{3/2}} \ (1+O(a_2)) \to \pi (1-t)^2 \ .
\end{align*}

(v) For $A_{1,\C}(a^{(2)},t), 0 < t < a_1=a_2 = \frac 1 {{\sqrt 2}}$, we are in the first case $0 \le  t < a_2$. Hence
$$A_{1,\C}(a^{(2)},t) = \frac \pi 2 (1-{\sqrt 2} t)^2 + 4 \int_{\frac 1 2 -\frac t {{\sqrt 2}}}^{\frac 1 2 +\frac t {{\sqrt 2}}}
\arccos(\frac 1 {{\sqrt 2} t} - \frac{1/2-t^2}{{\sqrt 2} t} \frac 1 r) \ r dr \ . $$
Let $B:=\frac 1 {{\sqrt 2} t}$, $B>1$ and $C:=\frac{1/2-t^2}{{\sqrt 2} t}$. Then
\begin{align*}
\int \arccos(B - \frac C r) \ r dr & = \frac{r^2} 2 \arccos(B - \frac C r) - \frac C {2(B^2-1)} \sqrt{r^2(1-B^2)+2CBr-C^2} \\
& + \frac{C^2 B}{2(B^2-1)^{3/2}} \arctan(\frac{r(B^2-1)-BC}{\sqrt{B^2-1} \sqrt{r^2(1-B^2) + 2 C B r -C^2}} ) \ .
\end{align*}
For $r=r_+ := \frac 1 2 + \frac t {{\sqrt 2}}$ or $r=r_- := \frac 1 2 - \frac t {{\sqrt 2}}$ we have that $B - \frac C r = +1$ or $=-1$, respectively, and
$\sqrt{r^2(1-B^2)+2CBr-C^2} = 0$. Also $r(B^2-1)-BC = \pm \frac{1-2 t^2}{2 {\sqrt 2} t}$ has the same sign as $r_{\pm}$. Further, $\frac{C^2 B}{2 (B^2-1)^{3/2}} = \frac 1 8 \sqrt{1-2 t^2}$. Therefore with $\arctan(\infty)-\arctan(-\infty) = \pi$
$$\int_{r_-}^{r_+}\arccos(B-\frac C r) \ r dr = -\frac {\pi} 2 (\frac 1 2 -\frac t {{\sqrt 2}})^2 +\frac{\pi} 8 \sqrt{1-2 t^2} = \frac {\pi} 8 (\sqrt{1-2t^2} - (1-{\sqrt 2} t)^2 ) \ , $$
so that $A_{1,\C}(a^{(2)},t) = \frac{\pi} 2 \sqrt{1-2 t^2}$. \\

(vi) In the case of $a^{[t]} = (\sqrt{1-t^2},t)$, $t = a_2 < a_1$, $t < \frac 1 {{\sqrt 2}}$
$$A_{1,\C}(a^{[t]},t) = \frac 2 {t^2}  \int_0^{\frac{2t}{t+\sqrt{1-t^2}}} \arccos( \frac t {\sqrt{1-t^2}} + \frac{1-2 t^2}{2t \sqrt{1-t^2} } r) \ r dr \ . $$
This can be evaluated similarly as in (v), using with $B:=\frac t {\sqrt{1-t^2}}$ and $C:=\frac{1-2 t^2}{2t \sqrt{1-t^2}}$ that
$$\int \arccos(B + C r) \ r dr = (\frac {r^2} 2-\frac{1+2B^2}{4 C^2}) \arccos(B+C r) +\frac{3B-Cr}{4C^2} \sqrt{1-B^2-2BCr-C^2 r^2} \ , $$
which yields the result stated in Proposition \ref{prop7.1} after some calculation.
\end{proof}

\vspace{0,5cm}

The explicit formulas for $a^{(1)}$, $a^{(2)}$ and $a^{[t]}$ give
\begin{align*}
A_{1,\C}(a^{(2)},t) & < A_{1,\C}(a^{(1)},t) \text{\quad iff \quad} t \in (0,0.3370) \cup (0.6947,\frac 1 {{\sqrt 2}}) \ , \\
A_{1,\C}(a^{[t]},t) & < A_{1,\C}(a^{(1)},t) \text{\quad iff \quad} t \in (0,0.2163) \cup (0.6932,\frac 1 {{\sqrt 2}}) \ , \\
A_{1,\C}(a^{[t]},t) & < A_{1,\C}(a^{(2)},t) \text{\quad iff \quad} t \in (0.4371,\frac 1 {{\sqrt 2}}) \ .
\end{align*}

Numerical evaluation of the general formula in Proposition \ref{prop7.1} yields the following conjecture for the extrema of $A_{1,\C}(\cdot,t)$ :
$$ Maximum :  \left\{\begin{array}{c@{\;}l}
a^{(1)} , &   \ 0 \le t < 0.179 \\
a^{(2)} , &   0.487 < t < 0.694 \\
a^{(1)} , &   \frac 1 {{\sqrt 2}} < t
\end{array}\right\}   , \
Minimum : \left\{\begin{array}{c@{\;}l}
a^{(2)} , &  \  0 \le t < 0.337 \\
a^{(1)} , &  0.337 \le t < 0.684
\end{array}\right\}  \ . $$
By numerical evidence, for $0.18 < t < 0.486$, the maximal vector $a_{max}$ seems to move continuously from $a^{(1)}$ to $a^{(2)}$ with $a_{max} \ne a^{[t]}$, except for possibly one value $\tilde{t}$, $\tilde{t} \simeq 0.33$; for $0.694 < t < \frac 1 {{\sqrt 2}}$, $a_{max}$ is close to $a^{(1)}$, but not equal. In the minimum case, for $0.684 < t < \frac 1 {{\sqrt 2}}$, the minimal vector seems to be close to $a^{(2)}$, but not equal. \\
Comparing this with the real $2$-dimensional case, mentioned before Proposition \ref{prop2}, $a^{[t]}$ does not give the maximum for $0 < t < \frac 1 {{\sqrt 2}}$, except possibly for $\tilde{t}$. The minimum case seems similar, with $1 - \frac 1 {{\sqrt 2}}$ being replaced by $0.337$, except for a change very close to $\frac 1 {{\sqrt 2}}$.

\vspace{0,2cm}

\end{document}